\documentclass[12pt]{amsart}
\usepackage[english]{babel}
\usepackage{graphicx}
\usepackage{color}
\usepackage{amsmath,amssymb,hyperref,srcltx}
\usepackage[inline]{showlabels}

\usepackage[shortlabels]{enumitem}
\usepackage{xifthen}
\newtheorem{theorem}{Theorem}
\newtheorem{definition}{Definition}
\newtheorem{proposition}{Proposition}
\newtheorem{lemma}{Lemma}
\newtheorem{corollary}{Corollary}
\newtheorem{exam}{Example}

\newtheorem{exams}{Examples}

\newtheorem{rmk}{Remark}
\newenvironment{remark}{\begin{rmk}\rm}{\end{rmk}}
\newtheorem{notat}{Notation}

\renewcommand{\ne}{\not =}

\newcommand{\diff}[2]{\mbox{{\rm Diff}{${\,}_{#1}({\mathbb C}^{#2},0)$}}}
\newcommand{\diffh}[2]{\mbox{$\widehat{\rm Diff}{{\,}_{#1}({\mathbb C}^{#2},0)}$}}
\newcommand{\cn}[1]{\mbox{\rm{(}${\mathbb C}^{#1},{\bf 0}$\rm{)}}}
\newcommand{\Xn}[2]{\mbox{${\mathfrak X}_{#1}\cn{#2}$}}
\newcommand{\Xf}[2]{\mbox{$\hat{\mathfrak X}_{#1}\cn{#2}$}}
\title[Completely Integrable Foliations]{Completely Integrable Foliations: Singular Locus, Invariant Curves and Topological Counterparts}
\author{Javier Ribón}
\address{Instituto de Matem\'{a}tica e Estat\'\i stica \\
Universidade Federal Fluminense\\
Campus do Gragoat\'a\\
Rua Marcos Valdemar de Freitas Reis s/n, 24210\,-\,201 Niter\'{o}i, Rio de Janeiro - Brasil }
\thanks{e-mail address: jribon@id.uff.br}
\thanks{MSC class. Primary: 32S65, 34A05; Secondary: 34C45,  32M25, 34M65, 34M35}
\begin{document}

\bibliographystyle{plain}

\maketitle
\tableofcontents  
\section*{Abstract} 
We study codimension $q \geq 2$ holomorphic foliations 
defined in a neighborhood of a point $P$ of a complex manifold that are  completely integrable, 
i.e.  with $q$ independent meromorphic first integrals.  
We show that either $P$ is a regular point, a non-isolated singularity or 
there are infinitely many 
invariant analytic varieties through $P$ of the same dimension as the foliation, the so called separatrices. 
Moreover, we see that this phenomenon is of topological nature. 

Indeed, we introduce topological counterparts of completely integrable local holomorphic foliations
and tools, specially the concept of total holonomy group, to build 
holomorphic first integrals if they have isolated separatrices. 
As a result, we provide a topological characterization 
of completely integrable non-degenerated elementary isolated singularities of vector fields 
with an isolated separatrix.
\section{Introduction}
We study completely integrable holomorphic foliations or, more precisely, foliations with 
the highest possible number of independent holomorphic (or meromorphic) first integrals.
Such number is the codimension of the foliation. In order to fix the setting of the paper, 
let us introduce the definition of completely integrable vector field.  
\begin{definition}
\label{def:complete}
A holomorphic foliation $\mathcal{F}$ of codimension $q$ is
completely integrable (CI) if there exist meromorphic first integrals $f_1, \ldots, f_q$  
of $\mathcal{F}$  such that $df_1\wedge \ldots \wedge df_q \not \equiv 0$.
We say that $(f_1, \ldots, f_q)$ is a complete first integral of $\mathcal{F}$. 
We say that $\mathcal{F}$ is $(CI)_{\mathcal{O}}$ if 
the definition holds for holomorphic $f_1, \ldots, f_q$.
\end{definition}

For example if $\mathcal{F}$ is given by a germ of vector field $X$, we are requiring
$X(f_1) = \ldots = X(f_{n-1}) =0$.  
The property $df_1\wedge \ldots \wedge df_q \not \equiv 0$ is the condition of functional independence of the first integrals.

This paper is devoted to the study of complete integrability for germs of holomorphic foliations and its relation with their topological properties.  
This program was initiated by Mattei and Moussu, who studied holomorphic complete integrability  $(\mathrm{CI})_{\mathcal{O}}$ 
in their seminal paper \cite{MaMo:Aen} for singularities of foliations of codimension $q=1$.
They proved that the existence of a non-trivial formal first integral is equivalent to the existence of an analytic one
and that the $(\mathrm{CI})_{\mathcal{O}}$ property is a topological invariant.
In spite of this, the CI property is not a topological invariant even for $q=1$ \cite[Suzuki]{Suzuki(int.pre)}. 
This paper treats the case $q \geq 2$; this is 
a current topic of research in which new phenomena appear. 
For instance $(\mathrm{CI})_{\mathcal{O}}$ is not a topological invariant for $q=2$ 
\cite[Pinheiro-Reis]{Pinheiro-Reis:topological}.

Let us remark that singularities of one dimensional complete integrable foliations have been studied from the point of view of 
normal forms \cite[Zhang]{Zhang:ana_norm1, Zhang:ana_norm2} \cite[Llibre-Pantazi-Walcher]{Llibre-Pantazi-Walcher:first_integrals}.

A priori, we could expect completely integrable foliations to be simple since 
the leaves are generically connected components of level sets of a complete first integral.
Moreover, a complete first integral determines the foliation. 
Nevertheless, the properties of CI foliations are subtle and it is not straightforward 
to obtain the geometrical properties of the foliation from a complete first integral.
 
As an example, consider $n=3$, $f_1 (x,y,z) = xz$, $f_{2} (x,y, z) = y z$. 
The $1$-form   
\[ \omega := df_{1} \wedge df_{2}= z (z dx \wedge dy + y dx \wedge dz - x dy \wedge dz) \]
determines the $1$-dimensional foliation ${\mathcal F}$ induced by the germ of vector field
\begin{equation}
\label{equ:ex1}
 X = x \frac{\partial}{\partial x} + y \frac{\partial}{\partial y} - z \frac{\partial}{\partial z}. 
 \end{equation}
that has isolated singularity whereas $\mathrm{cod} (\mathrm{Sing}(\omega))=1$.
This discrepancy between singular sets is independent of the choice of first integrals. 
Indeed, if we have two germs $\omega_1, \omega_2$ of holomorphic $1$-forms such that 
$\omega_{1} (X) = \omega_{2} (X) = 0$ and $\omega_{1} \wedge \omega_2 \neq 0$, 
we obtain that $\omega_1 \wedge \omega_2$ is a multiple of $z^{-1} \omega$ and thus singular at ${\bf 0}$. 
It is a well-known result that in such a case ${\bf 0}$ is a non-isolated singular point of
$\omega_1 \wedge \omega_2$  
\cite[Lemma 3.1.2]{Medeiros:singular}  \cite{Mal:Frob2} and hence we obtain 
$\mathrm{cod} ({\mathrm Sing}(\omega_{1} \wedge \omega_{2}))=1$.

We are interested in  the geometrical properties of germs of CI foliations ${\mathcal F}$.
First, we study two geometrically relevant sets associated to ${\mathcal F}$, namely 
the singular set $\mathrm{Sing}({\mathcal F})$ and the set of separatrices 
(the integral varieties of ${\mathcal F}$ whose closure contains ${\bf 0}$).
Next theorem reveals their transcendence, 
showing that at least one of the sets is ``big".
\begin{theorem}
\label{teo:dicritical}
Let $\mathcal{F}$ be a completely integrable germ of holomorphic foliation of 
codimension $q \geq 2$ 
in $\cn{n}$ with an isolated singularity at the origin.  
Then, there are infinitely many invariant analytic varieties of dimension
$\dim (\mathcal{F})$ through the origin.
\end{theorem}
An alternative statement is that if
if ${\mathcal F}$ is a CI germ that is singular at ${\bf 0}$ and $q \geq 2$, 
we obtain that either $\dim (\mathrm{Sing}({\mathcal F})) \geq 1$ or there are infinitely
many separatrices.  
For instance, in example \eqref{equ:ex1} the origin is an isolated singularity  and there are infinitely many separatrices, 
indeed any line through the origin in the plane $z=0$ is a separatrix. 
Of course, there are also examples of CI foliations with finitely many separatrices and non-isolated singularities 
(cf. Remark \ref{rem:tci_non_iso}).

Theorem \ref{teo:dicritical} is a result intrinsic of codimension $q \geq 2$. 
Indeed, for $q=1$, there are germs of function $f$, for instance $f(x,y)=xy$, such that 
$f=0$ has an isolated singularity. In such a case $df =0$ has finitely many separatrices, namely 
the irreducible components of $f=0$.

Theorem \ref{teo:dicritical} extends a theorem of 
Pinheiro and Reis \cite{Pinheiro-Reis:topological} 
where they suppose that $n=3$, $\mathcal{F}$ is $(\mathrm{CI})_{\mathcal{O}}$ and 
made additional hypothesis on the existence of reduction of singularities 
after a single blow-up. Let us remark that their techniques rely on desingularization theory and index arguments whereas
ours are of topological nature. 

Next result is an immediate corollary of Theorem \ref{teo:dicritical}.
\begin{corollary}
Consider germs $f_1, \hdots, f_q$ of holomorphic function defined in a neighborhood of ${\bf 0}$ in ${\mathbb C}^{n}$
such that $(f_1, \hdots, f_q)^{-1} ({\bf 0})$ has pure codimension $q$. 
Then, either ${\mathcal F}$ is regular at ${\bf 0}$ or $\dim (\mathrm{Sing}({\mathcal F})) \geq 1$, where 
${\mathcal F}$ is the foliation with complete first integral $(f_1, \hdots, f_q)$.
\end{corollary}
Indeed, the hypothesis on $(f_1, \hdots, f_q)^{-1} ({\bf 0})$ implies that ${\mathcal F}$ has finitely many separatrices.

\strut

Our techniques are of topological kind and as a consequence of their flexibility 
we can give topological versions of our results. 
We introduce the topological counterpart of completely integrable foliations.
\begin{definition}
\label{def:tci}
Let $\mathcal{F}$ be a germ of holomorphic foliation 
in $\cn{n}$ with isolated singularity. 
We say that ${\mathcal F}$ is topologically completely integrable (TCI) in $\overline{\operatorname{B}}(\mathbf{0}; \epsilon_{0})$
if ${\mathcal F}$ has a holomorphic representative 
defined in a neighborhood of $\overline{\operatorname{B}}(\mathbf{0}; \epsilon_{0})$ and 
\begin{itemize}
\item $\mathrm{Sing}({\mathcal F}) \cap \overline{\operatorname{B}}(\mathbf{0}; \epsilon_{0}) = \{ {\bf 0 } \}$;
\item every leaf $\mathcal{L}$ of the foliation $\mathcal{F}^{\epsilon_{0}}$
induced by $\mathcal{F}$ in 
$\operatorname{B}(\mathbf{0}; \epsilon_{0})$ is closed in 
$\operatorname{B}(\mathbf{0}; \epsilon_{0}) \setminus \{ \bf{0} \}$; 
\item given any closed leaf $\mathcal{L}$ of 
$\mathcal{F}^{\epsilon_{0}}$, its holonomy group is finite;
\item $\mathcal{L}$ is transverse to 
$\partial\operatorname{B}(\mathbf{0}; \epsilon_{0})$ for
any leaf $\mathcal{L}$ of $\mathcal{F}^{\epsilon_{0}}$ 
with $\bf{0} \in \overline{\mathcal{L}}$;
\end{itemize}
We say that ${\mathcal F}$ is topologically completely integrable (TCI) if in some coordinate system centered at ${\bf 0}$, 
given any $\epsilon_{0} >0$ there exists 
$\epsilon \in (0, \epsilon_{0})$ such that ${\mathcal F}$ is topologically integrable
in $\overline{\operatorname{B}}(\mathbf{0}; \epsilon)$.
\end{definition}
Let us remark that the choice of balls centered at ${\bf 0}$ as domains of study is unimportant. 
We could consider more general domains with smooth boundaries.

All conditions in the definition of TCI are of topological nature except transversality. 
Anyway, TCI is inherently a topological concept.
Indeed, we can give an entirely topological definition of TCI by replacing 
transversality with topological transversality in the fourth bullet point of Definition \ref{def:tci}.
Our constructions can be carried out with any of the definitions of TCI. 
We consider Definition \ref{def:tci} for simplicity.

The next theorem is a topological version of Theorem \ref{teo:dicritical}.
\begin{theorem}
\label{teo:infinite}
Let $\mathcal{F}$ be a germ of holomorphic foliation of 
codimension $q \geq 2$ 
in $\cn{n}$ with an isolated singularity at the origin. Suppose that
there exists an open ball $\operatorname{B}(\mathbf{0}; \epsilon)$ such that 
\begin{itemize}
\item every leaf $\mathcal{L}$ of the foliation $\mathcal{F}^{\epsilon}$ 
induced by $\mathcal{F}$ in 
$\operatorname{B}(\mathbf{0}; \epsilon)$ is closed in 
$\operatorname{B}(\mathbf{0}; \epsilon) \setminus \{ \bf{0} \}$ and
\item given any closed leaf $\mathcal{L}$ of 
$\mathcal{F}^{\epsilon}$, its holonomy group is finite.
\end{itemize}
Then, there are infinitely many invariant analytic varieties of dimension
$\dim (\mathcal{F})$ through the origin.
\end{theorem}
This result implies that the conditions in the Mattei-Moussu's topological criterium for the existence of 
holomorphic first integrals for germs of holomorphic foliations of codimension $q=1$ \cite{MaMo:Aen}, namely 
closedness of leaves, finiteness of holonomy groups and the existence of finitely many separatrices never happen 
for codimension $q \geq 2$. 
Let us remark that even if finiteness of holonomy groups is not explicitly required in the criterium, closedness of leaves
implies such a property for $q=1$, something that is no longer valid for $q=2$ in general.

Indeed, we will see that both hypotheses in Theorem \ref{teo:infinite}
are satisfied by a CI germ of foliation and thus Theorem \ref{teo:dicritical}
becomes a corollary of Theorem \ref{teo:infinite}.
Moreover,  a vector field satisfying the hypothesis of 
Theorem \ref{teo:infinite} and with finitely many separatrices is topologically completely integrable and 
hence it suffices to understand such a case to show Theorem \ref{teo:infinite}.
 
 \strut
 
 We explore the connection between topological complete integrability and complete integrability. 
 Indeed, we show that TCI implies $(\mathrm{CI})_{\mathcal{O}}$ if there exists an isolated separatrix. 
 Moreover, the TCI property provides relevant geometrical properties, since, for instance,    it implies the existence of 
 a dicritical variety of  codimension $1$.
\begin{theorem}
\label{teo:exist_cod1}
Let $\mathcal{F}$ be a germ of holomorphic foliation 
of codimension $q \geq 2$  that is TCI in 
$\overline{\operatorname{B}}(\mathbf{0}; \epsilon_{0})$
and such that $\mathcal{F}^{\epsilon_{0}}$ has an isolated separatrix.
Then $\mathcal{F}$ is $(CI)_{\mathcal{O}}$ and there exists a $\mathcal{F}$-dicritical  variety
of codimension $1$, i.e. an invariant variety $S$ in which all leaves $\mathcal{L}$ of 
$\mathcal{F}^{\epsilon_{0}}$ through points in $S \setminus \{ \bf{0} \}$ are 
non-closed in $\operatorname{B}(\mathbf{0}; \epsilon_{0})$. 
\end{theorem}
In order to produce holomorphic first integrals, we can not apply 
the Theorem of Frobenius with Singularities \cite{Mal:Frob2}. Indeed, even if 
 $\mathrm{cod} ({\mathrm Sing}({\mathcal F})) \geq 3$, the foliation ${\mathcal F}$ 
 does not necessarily satisfy its hypotheses. It can happen that 
 $\mathrm{cod} ({\mathrm Sing}(\omega_{1} \wedge \ldots \wedge \omega_{q}))=1$
 as in  \eqref{equ:ex1} for any $q$-uple of $1$-forms $\omega_1, \ldots, \omega_q$
 determining ${\mathcal F}$.
 Moreover, the hypotheses of the Theorem of Frobenius with Singularities never hold true for
 (singular at ${\bf 0}$) foliations of dimension $1$ since 
 $\mathrm{cod} ({\mathrm Sing}(\omega_{1} \wedge \ldots \wedge \omega_{n-1})) \leq 2$ in such a case
\cite[Lemma 3.1.2]{Medeiros:singular}.

We exploit the transversality of separatrices to obtain a complete first integral. 
Let us remark that ideas of this kind were used by Moussu to provide a topological proof of the topological criterium of existence of 
holomorphic first integrals for germs of codimension $1$ holomorphic foliations \cite{Moussu:exist_first_integrals}.

We introduce a relevant concept on its own: the  
{\it total holonomy group} (Definition \ref{def:total}). 
It is, essentially, a pseudo-group of holonomy that turns out to be a finite group of biholomorphisms with a common fixed point
for the setting of Theorem \ref{teo:exist_cod1}. Invariant functions by the action of this group induce
first integrals of the foliation. The total holonomy group can be defined for certain 
CI foliations with non-isolated singularities (cf. Remark \ref{rem:tci_non_iso}, where we consider 
a germ of codimension $2$ foliation in ${\mathbb C}^{4}$) and it will be fundamental 
for their study,  that will be treated in future work in collaboration with R. Rosas.
It is great news that the scope of our techniques is not limited to isolated singularities
since such a property is in some cases very restrictive, 
for example, it is not known if there is a germ of codimension two foliation in ${\mathbb C}^{4}$ with
isolated singularity (cf. \cite{Cerveau-Lins_Neto:cod2}).

Theorem \ref{teo:exist_cod1} can be greatly improved in the case of foliations of dimension $1$ with an isolated 
separatrix. Then we characterize all germs of foliations that are topologically 
completely integrable. 

\begin{theorem}
\label{teo:characterization}
Let $\mathcal{F}$ be a germ of one-dimensional  foliation in $\cn{n}$ ($n \geq 3$).
Then $\mathcal{F}$
is analytically conjugated to the foliation given by a germ of vector field of the form 
\begin{equation}
\label{equ:normal_form}
 \sum_{j=1}^{n} \lambda_{j} x_j \frac{\partial}{\partial x_j}, \ \  
\lambda_1, \ldots, \lambda_{n-1} \in {\mathbb Z}^{-}, \ \ \lambda_{n} \in {\mathbb Z}^{+},   
\end{equation} 
if and only ${\mathcal F}$ is topologically completely integrable and has an isolated separatrix.
\end{theorem}
There are several instances in the literature in which transversality 
of a foliation ${\mathcal F}$ of dimension or codimension $1$ with the boundary of a domain 
imposes strong constrains
on ${\mathcal F}$ \cite{Ito:Poincare-Bendixson}  \cite{Brunella-Sad:hol_in_convex} \cite{Brunella:transverse}
\cite{Ito-Scardua:fol_cod_1_transverse} \cite{Ito-Scardua:non-existence-cod1} 
\cite{Bracci-Scardua:vector_fields_transverse} \cite{Ito-Scardua:non-existence-Morse} 
as it happens in Theorem \ref{teo:characterization}.
In contrast with such results, we do not impose transversality in the whole boundary and just at the separatrices
making our approach closer to Moussu's in  \cite{Moussu:exist_first_integrals}.

We discuss CI foliations in section \ref{sec:ci}. Section \ref{sec:infinite} is devoted to the proof of Theorem \ref{teo:infinite}.
The proof of Theorem \ref{teo:exist_cod1} and in particular of the existence of a dicritical hypersurface is carried out in 
section \ref{sec:cod1}. It is introduced also the concept of {\it total holonomy group}. 
We provide a characterization of TCI germs of vector field, with an isolated separatrix,  in section  \ref{sec:tci}.

\section{Completely integrable foliations}
\label{sec:ci}
First, let us introduce the results that allow to reduce 
the proof of Theorem \ref{teo:dicritical} to the proof 
of Theorem \ref{teo:infinite}.
\subsection{Closedness of Leaves}
The purpose of this subsection is proving that all leaves of a germ of completely integrable foliation
are closed off $ \mathrm{Sing}(\mathcal{F})$. 
This is obvious for leaves $\mathcal{L}_P$ through generic points
since in such a case $\overline{\mathcal{L}}_P$ is an irreducible component of the level set 
of a complete first integral $(f_1, \ldots, f_q)$ containing $P$ such that 
$\dim ( \overline{\mathcal{L}}_P) = \dim (\mathcal{F})$ and 
$\overline{\mathcal{L}}_P \setminus \mathrm{Sing} (\mathcal{F}) = \mathcal{L}_P$.
Since it is not true, in general, that any level set of $(f_1, \ldots, f_q)$ has pure dimension $\dim (\mathcal{F})$, 
we need to deal with leaves contained in level sets of higher dimension.
The next result is the first step in this task.  
\begin{proposition}
\label{pro_basica}
Consider a normal closed analytic variety $S$ defined in a complex manifold $M$.
Let $f, g$ be meromorphic functions defined in some open subset $U$ of $S$
such that $df \wedge dg \not \equiv 0$.
Let $H$ be an irreducible closed analytic subvariety of $S \cap U$ of codimension $1$. 
Then, there exists $h \in \mathcal{K}$ that is non-constant on $H$, where $\mathcal{K}$ is the field
generated by $f$, $g$ and the constant functions.
\end{proposition}
An analogous result for the case of holomorphic functions $f, g$ is proved in \cite{Pinheiro-Reis:topological}.
\begin{proof} 
  Let us reason by contradiction, 
assuming that every $h \in \mathcal{K}$ is constant on $H$.
Since $S$ is normal and thus $\mathrm{cod} (\mathrm{Sing}(S)) \geq 2$, 
we can define the annulation order $\nu_H(a)$ of a in $H$ for any 
weakly holomorphic function $a$.
  Given a meromorphic germ $a/b$ we define the annulation order 
$\nu_H(a/b)=\nu_H(a)-\nu_H(b)$. Let $r\geq 1$ be the minimum of the positive annulation orders of 
functions in $\mathcal{K}$. Choose a 
meromorphic first integral $\Psi \in \mathcal{K}$ with $\nu_H(\Psi)=r$. 
Let $\Upsilon \in \mathcal{K}$ be 
another meromorphic first integral  with $\nu_H(\Upsilon)\geq 0$. 
We are assuming that $\Upsilon$ is constant along $H$, having a fixed value  $c_0\in \mathbb{C}$. 
There is an integer number $m_0\geq 1$ such that
$$
\nu_H(\Upsilon-c_0)=m_0r.
$$
Indeed, otherwise, we have that $\nu_H(\Upsilon-c_0)=qr+s$, with $0<s<r$, and then the properties 
$(\Upsilon-c_0) \Psi^{-q} \in \mathcal{K}$ and
$$
\nu_H((\Upsilon-c_0) \Psi^{-q})=s,
$$
contradict the minimality of $r$. Denote $\Upsilon_1=(\Upsilon-c_0)\Psi^{-m_0}$. 
Then $\Upsilon_1 \in \mathcal{K}$ and $\Upsilon_1\vert _H$ is constant and equal to $c_1\ne 0$. 
Now, we repeat the procedure with $\Upsilon_1-c_1$, noting that $\nu_H(\Upsilon_1-c_1)>0$ and we write
$
\Upsilon_2=(\Upsilon_1-c_1)\Psi^{-m_1}
$.
We obtain
\begin{eqnarray*}
\Upsilon &=& c_0+\Upsilon_1\Psi^{m_0} \\
&=& c_0+c_1 \Psi^{m_0} + \Upsilon_2\Psi^{m_0+m_1}      \\
&=& c_0+c_1 \Psi^{m_0} + c_2 \Psi^{m_0+m_1}  + \Upsilon_3 \Psi^{m_0+m_1+m_2}    \\
&=& \cdots    \\
&=& c_0+\sum_{j=1}^k c_j \Psi^{m_0+m_1+\cdots+m_{j-1}}+  \Upsilon_{k+1}\Psi^{m_0+m_1+\cdots+m_{k}} ,
\end{eqnarray*}
where $\nu_H (\Upsilon_j) \geq 0$, for any $j\geq 1$.
If we compute $d\Upsilon\wedge d\Psi$ with the formulas above, we conclude that
$$
\nu_H(d\Upsilon\wedge d\Psi)\geq (m_0+m_1+\cdots+m_{k})\nu_H(\Psi)=(m_0+m_1+\cdots+m_{k})r,
$$
for any $k\geq 0$. Therefore $d\Upsilon\wedge d\Psi \equiv 0$. Moreover, such result holds for any 
$\Upsilon \in \mathcal{K}$ since $\nu_{H} (\Upsilon)<0$ implies $\nu_{H} (\Upsilon^{-1})>0$ and hence
\[ d\Upsilon\wedge d\Psi \equiv  - \Upsilon^{2} (d\Upsilon^{-1}\wedge d\Psi) \equiv 0. \]
We deduce $d f \wedge  d\Psi \equiv 0$ and $d g \wedge  d\Psi \equiv 0$.
This implies that $df\wedge dg \equiv 0$, contradiction.
\end{proof}
 The previous result has independent interest since it implies that, in our setting, 
 meromorphic first integrals generically separate points in a hypersurface.
  \begin{corollary}
 Let $\mathcal{F}$ be a germ of CI foliation of codimension $q \geq 2$ in $\cn{n}$.
 Consider an irreducible germ $H$ of hypersurface. 
 Then there is a meromorphic first integral $\Psi$ of $\mathcal{F}$ such that 
 $\Psi_{|H}$ is non-constant.
 \end{corollary}
 Next lemma will be used to show that, for CI foliations of codimension $q \geq 2$, there are infinitely many
 invariant hypersurfaces containing a fixed point $P \not \in \mathrm{Sing} (\mathcal{F})$.
\begin{lemma}
\label{lem:pure_meromorphic}
Consider an irreducible closed analytic variety $S$ defined in a complex manifold $M$.
Let $\mathcal{K}$ be a subfield of the field of meromorphic functions containing the constant functions and defined
in some open subset $U$ of $S$.
Suppose there exist $f, g \in \mathcal{K}$ such that $df \wedge dg \not \equiv 0$. 
Then, given $P \in S \cap U$ and an irreducible component $S'$ of the germ $(S,P)$, there exists $h \in \mathcal{K}$ 
such that the neither the
germ of $h$ at $(S',P)$ nor its inverse $h^{-1}$ is weakly holomorphic. 
\end{lemma}
\begin{proof}
Consider the normalization $\pi: \hat{S} \to S$. 
Let $\hat{P}$ be the point in $\pi^{-1}(P)$
such that $P$ belongs to the closure of $\pi^{-1} (S' \setminus \mathrm{Sing}(S))$.
Notice that since $\hat{S}$ is normal, the germ $(\hat{S}, \hat{P})$ is irreducible.
There is a bijective 
correspondence between the weakly holomorphic (resp. meromorphic) functions of $U$
and the holomorphic (resp. meromorphic) functions of $\pi^{-1}(U)$.
Thus it suffices to find $h \in \pi^{*} (\mathcal{K})$ such that neither 
$h$ nor $h^{-1}$ is holomorphic at $\hat{P}$.
Therefore, modulo to replace $S$ and $P$ with $\hat{S}$ and $\hat{P}$ respectively, 
we can suppose that $S$ is normal.
 
Let us argue by contradiction, assume that $h$ is non-purely meromorphic at $P$
for any $h \in \mathcal{K}$. Up to replace $f$ with $f^{-1}$ if necessary, we can assume 
that $f$ is holomorphic at $P$. 
Moreover, we can suppose $f(P)=0$ up to replace 
$f$ with $f - f(P)$. Let $H$ be an irreducible component of   $f^{-1}(0)$
such that $(H,P)$ contains an irreducible component $H_P$ of $(f^{-1}(0),P)$.
There exists $\Psi \in \mathcal{K}$ such that $\Psi_{|H}$ is non-constant by Proposition \ref{pro_basica}.
Analogously as for $f$, we can suppose that $\Psi$ is holomorphic at $P$ and $\Psi (P) =0$.
Let $H_{P}'$ be an irreducible component of $(\Psi^{-1}(0),P)$.  Since 
$\nu_{H_P}(f/\Psi^{k})=\nu_{H_P}(f)>0$ and 
\[ \nu_{H_{P}'}(f/\Psi^{k}) = \nu_{H_{P}'}(f) - k \nu_{H_{P}'}(\Psi) \leq  \nu_{H_{P}'}(f) - k \] 
for $k \in \mathbb{Z}_{>0}$, it follows that 
$f / \Psi^{k}$ is purely meromorphic for 
any $k \in \mathbb{Z}_{>0}$ sufficiently big. 
\end{proof}
\begin{definition}
We denote by ${\operatorname{B}}_{\epsilon}$ and $\overline{\operatorname{B}}_{\epsilon}$
the open and closed balls of center ${\bf 0}$ and radius $\epsilon$ respectively.
\end{definition}
\begin{definition}
\label{def:fe}
We denote by ${\mathcal F}^{\epsilon}$ the foliation induced by $\mathcal{F}$ in ${\operatorname{B}}_{\epsilon}$.
Given $P \in {\operatorname{B}}_{\epsilon} \setminus \mathrm{Sing} (\mathcal{F})$, 
we denote by $\mathcal{L}_{P}^{\epsilon}$ the leaf of $\mathcal{F}^{\epsilon}$ through $P$. 
Sometimes we drop the superindex $\epsilon$ if is implicit.
\end{definition}
Finally, we can show the main result in this subsection.
\begin{proposition}
  \label{prop:closedleaves}
  Let $\mathcal{F}$ be a CI germ of foliation.
  Then there exists $\epsilon_{0}>0$ and a representative of 
  $\mathcal{F}$ defined in an open neighborhood of 
   the closed ball $\overline{\operatorname{B}}_{\epsilon_{0}}$ such that for 
   any $P\in {\operatorname{B}}_{\epsilon_{0}} \setminus \mathrm{Sing}(\mathcal{F})$, 
   the  leaf ${\mathcal L}_{P}^{\epsilon_{0}}$ of $\mathcal{F}^{\epsilon_{0}}$ through $P$ is a closed analytic subset 
   of ${\operatorname{B}}_{\epsilon_{0}} \setminus \mathrm{Sing}(\mathcal{F})$.
   In particular, the closure of ${\mathcal L}_{P}^{\epsilon_{0}}$
   in  ${\operatorname{B}}_{\epsilon_{0}}$ 
   is a closed analytic set of dimension $\dim (\mathcal{F})$ such that 
   $\overline{\mathcal L}_{P}^{\epsilon_{0}} \setminus \mathrm{Sing}(\mathcal{F})= {\mathcal L}_{P}^{\epsilon_{0}}$.
\end{proposition}
\begin{proof}
Let $(f_1, \ldots, f_q)$ be a complete first integral of $\mathcal{F}$.
Consider $\epsilon_{0}>0$ such that $\mathcal{F}$ and $f_1, \ldots, f_q$ have
representatives defined in a neighborhood of $\overline{\operatorname{B}}_{\epsilon_{0}}$.

We claim that given any $1 \leq r \leq q$, there exist 
an $\mathcal{F}^{\epsilon_{0}}$-invariant irreducible closed analytic set $H_r$
of dimension $\dim (\mathcal{F}) + r$
in ${\operatorname{B}}_{\epsilon_{0}}$ containing $P$
and meromorphic first integrals $f_{1,r}, \ldots, f_{r,r}$ 
of $\mathcal{F}_{|H_r}$ 
such that $d f_{1,r}  \wedge \ldots \wedge d f_{r,r} \not \equiv 0$.
The result is obvious for $r=q$ since we consider $H_q = {\operatorname{B}}_{\epsilon_{0}}$
and $f_{j,q} = f_j$ for $1 \leq j \leq q$.
Let us show that if it holds for $r+1 \geq 2$ so it does for $r$. 

Consider the field $\mathcal{K}_{r+1}$ generated by the constant functions and $f_{1,r+1}, \hdots, f_{r+1,r+1}$.
Let $S'$ be an irreducible component of the germ $(H_{r+1}, P)$. Then there exists 
$\Psi \in \mathcal{K}_{r+1}$ such that 
$\Psi$ is purely meromorphic at $(S', P)$ by Lemma \ref{lem:pure_meromorphic}.
By considering irreducible components  
of level sets $\{ \Psi = cte \} \cap H_{r+1}$, we obtain infinitely many 
$\mathcal{F}^{\epsilon_{0}}$-invariant irreducible closed analytic sets $H$ of dimension $\dim (\mathcal{F}) + r$
that contain $P$. In particular we can choose one of them $H_r$ that is not contained in the zero set 
of $d f_{1,r+1}  \wedge \ldots \wedge d f_{r+1,r+1}$.
Thus there exists a subset of $r$ elements of $((f_{1,r+1})_{|H_r}, \ldots, (f_{r+1,r+1})_{|H_r})$
that is a complete first integral of $\mathcal{F}_{|H_r}^{\epsilon_{0}}$.
The proof of the claim is complete.

Let us fix $r=1$, so that $\mathcal{F}_{|H_1}^{\epsilon_{0}}$ is a codimension $1$ foliation with non-constant 
meromorphic first integral $f_{1,1}$.
The irreducible components of the levels of $f_{1,1}$ 
provide a partition of $H_{1} \setminus \mathrm{Sing}(\mathcal{F})$ in 
$\mathcal{F}^{\epsilon_{0}}$-invariant irreducible smooth closed analytic subvarieties
 of $H_{1} \setminus \mathrm{Sing}(\mathcal{F})$ 
 whose closures are irreducible closed analytic subvarieties of $H_1$ of dimension $\dim (\mathcal{F})$.
 This is the partition of $H_{1} \setminus \mathrm{Sing}(\mathcal{F})$ in leaves of $\mathcal{F}_{|H_1}^{\epsilon_{0}}$.
 Therefore, we obtain $\dim (\overline{\mathcal L}_{P} )=\dim (\mathcal{F})$ and 
 $\overline{\mathcal L}_{P} \setminus \mathrm{Sing}(\mathcal{F}) = \mathcal{L}_{P}$.
\end{proof}

\begin{remark}
It is clear that Proposition \ref{prop:closedleaves} also holds for any $\epsilon \in (0, \epsilon_{0}]$.
\end{remark}

\begin{definition}
\label{def:separatrix}
Consider a holomorphic foliation $\mathcal{F}$ defined in a neighborhood 
of the origin in $\mathbb{C}^{n}$. 
We say that a leaf $\mathcal{L}$ is a separatrix if $\overline{\mathcal L} \cap \mathrm{Sing} (\mathcal{F})  \neq \emptyset$.
\end{definition}
\begin{remark}
Under the hypotheses of Proposition \ref{prop:closedleaves}
and $\mathrm{Sing} (\mathcal{F}) = \{ \bf{0} \}$, 
the foliation $\mathcal{F}^{\epsilon_{0}}$ has two types of 
leaves, namely closed leaves in ${\operatorname{B}}_{\epsilon_{0}}$ and separatrices ${\mathcal L}$ such that 
$\overline{\mathcal L} = {\mathcal L} \cup \{ {\bf 0} \}$.
\end{remark}  
\subsection{Stability of leaves}
There are two topological hypotheses in Theorem \ref{teo:infinite}, namely closedness of leaves 
and finiteness of holonomy groups. Both are a consequence of the CI property. In the former case
we proved it in Proposition \ref{prop:closedleaves} whereas the latter property is treated in this subsection.

We will also show that any closed leaf of a CI foliation
has a fundamental system of invariant neighborhoods.  
Since leaves are not necessarily transverse to the boundary of the domain of definition, we need
the following definition in order to give a precise statement.
\begin{definition} 
Consider a holomorphic foliation $\mathcal{F}$ defined in a neighborhood of $\overline{\operatorname{B}}_{\epsilon}$.
Given a point $P \in \overline{\operatorname{B}}_{\epsilon}\setminus \mathrm{Sing}(\mathcal{F})$, we define 
{\em the leaf ${\mathfrak L}^\epsilon_P$ of $\mathcal{F}$ in  $\overline{\operatorname{B}}_{\epsilon}$ through $P$} by 
$$
{\mathfrak L}^\epsilon_P=\bigcap_{\epsilon < \epsilon' << \epsilon + 1}{\mathcal L}_P^{\epsilon'},
$$
where each ${\mathcal L}_P^{\epsilon'}$ is the leaf of 
$\mathcal{F}^{\epsilon'}$ through $P$. 
We will say that ${\mathfrak L}^\epsilon_P$ is a leaf of $\mathfrak{F}^{\epsilon}$.
We will say that the holonomy group of ${\mathfrak L}^\epsilon_P$ is finite if there exists 
$\epsilon'>\epsilon$ such that the image of the holonomy representation 
$\pi_{1} ({\mathcal L}_{P}^{\epsilon'}, P) \to \mathrm{Diff} (T, P)$, where 
$T$ is a transversal to ${\mathcal F}$ at $P$, is finite. 
\end{definition}
Note that $\mathrm{Diff} (T, P)$ is the group of germs of biholomorphisms at $(T,P)$ 
(cf. Definition \ref{def:formal_dif}).
\begin{remark}
Note that $Q \in {\mathfrak L}^\epsilon_P$ if and only if $P \in {\mathfrak L}^\epsilon_Q$ and in such a case 
${\mathfrak L}^\epsilon_P = {\mathfrak L}^\epsilon_Q$ holds.
Thus the leaves of $\mathfrak{F}^{\epsilon}$ provide a partition of 
$\overline{\operatorname{B}}_{\epsilon} \setminus \mathrm{Sing} (\mathcal{F})$.
\end{remark}
\begin{remark}
\label{rem:closed_implies_connected}
A priori a leaf ${\mathfrak L}^\epsilon_P$ is not necessarily connected. 
Nevertheless, it is connected if it is closed. In such a case, fixed $\epsilon'' > \epsilon$, 
there exists a compact neighborhood of ${\mathfrak L}^\epsilon_P$ in 
${\mathcal L}_{P}^{\epsilon''}$. Thus, $\overline{\mathcal L}_{P}^{\epsilon'}$ is compact 
and satisfies 
\[  {\mathcal L}_{P}^{\epsilon'} = 
\overline{\mathcal L}_{P}^{\epsilon'} \cap {\operatorname{B}}_{\epsilon'} \]
for $0 < \epsilon' - \epsilon <<1$. This implies  
${\mathfrak L}^\epsilon_P=\bigcap_{0< \epsilon' -\epsilon <<1} \overline{\mathcal L}_P^{\epsilon'}$ and hence
that ${\mathfrak L}^\epsilon_P$ is connected, thanks to classical results 
on continua topological spaces (cf. \cite[Theorem 28.2]{Willard:general_topology}).  
\end{remark}
\begin{remark}
\label{rem:aclosed_implies_connected}
Consider the case where $\overline{\operatorname{B}}_{\epsilon} \cap \mathrm{Sing}(\mathcal{F}) = \{ \bf{0} \}$.
Any leaf $\mathfrak{L}_{P}^{\epsilon}$  that is closed in 
$\overline{\operatorname{B}}_{\epsilon} \setminus  \{ \bf{0} \}$ satisfies that 
$\overline{\mathfrak L}_{P}^{\epsilon} \cap {\operatorname{B}}_{\epsilon}$ is a  closed analytic subset
of ${\operatorname{B}}_{\epsilon}$   by 
Remmert-Stein's Theorem (cf. \cite[Theorem K7]{Gunning}).
Moreover, we claim that $\overline{\mathfrak L}_{P}^{\epsilon}$ and ${\mathfrak L}_{P}^{\epsilon}$ are both connected.  
Indeed, we can proceed as in Remark \ref{rem:closed_implies_connected} to show 
\begin{equation}
\label{equ:leaves_closed} 
\overline{\mathfrak L}^\epsilon_P \setminus \{ {\bf 0} \} = 
\bigcap_{0< \epsilon' -\epsilon <<1} (\overline{\mathcal L}_P^{\epsilon'} \setminus \{ {\bf 0} \}) =
\bigcap_{0< \epsilon' -\epsilon <<1} {\mathcal L}_P^{\epsilon'} =  {\mathfrak L}^\epsilon_P. 
\end{equation}
Thus $\overline{\mathfrak L}^\epsilon_P = \bigcap \overline{\mathcal L}_P^{\epsilon'}$
is a continuum that is an analytic set in a neighborhood of $\bf{0}$. 
It follows that $\mathfrak{L}_{P}^{\epsilon}$ is connected. Moreover if 
${\mathcal L}_P^{\epsilon'}$ is closed in 
${\operatorname{B}}_{\epsilon'} \setminus  \{ \bf{0} \}$  for some $\epsilon' > \epsilon$
then Equation \eqref{equ:leaves_closed}  still holds and ${\mathfrak L}^\epsilon_P$ is 
closed in $\overline{\operatorname{B}}_{\epsilon} \setminus  \{ \bf{0} \}$ and connected.
\end{remark}

The next result is a version, in our setting, of Reeb local stability theorem \cite{Wu-Reeb,Camacho-Lins_Neto:foliations}.

\begin{proposition}
\label{pro:stability}
  Consider a holomorphic foliation $\mathcal{F}$ defined in a neighborhood of the closed ball 
  $\overline{\operatorname{B}}_{\epsilon}$. 
  Let ${\mathfrak L}^\epsilon_P$ be a leaf of $\mathfrak{F}^{\epsilon}$ such that 
  $\overline{\mathfrak L}^\epsilon_P = {\mathfrak L}^\epsilon_P$ and whose holonomy group is finite.
  Then there is a fundamental system of  $\mathfrak{F}^{\epsilon}$-invariant neighborhoods of 
  ${\mathfrak L}^\epsilon_P$ in $\overline{\operatorname{B}}_{\epsilon}$.
\end{proposition}
\begin{proof}
Denote ${\mathfrak L} = {\mathfrak L}^\epsilon_P$. 
  It suffices to show that given
an open neighborhood $V$ of $\mathfrak{L}$ in $\mathbb{C}^{n}$,
there exists an $\mathfrak{F}^{\epsilon}$-invariant neighborhood of $\mathfrak{L}$ in
$\overline{\operatorname{B}}_{\epsilon}$ contained in $V$. 
Let us consider $\epsilon''$ with $0<\epsilon'' - \epsilon<<1$. 
There exists a compact neighborhood of ${\mathfrak L}$ in ${\mathcal L}_{P}^{\epsilon''}$. 
As a consequence, the leaf ${\mathcal L}_{P}^{\epsilon'}$ is closed 
in ${\operatorname{B}}_{\epsilon'}$ if 
$\epsilon < \epsilon' << \epsilon''$. Fix such $\epsilon'$ such that the holonomy group of 
$\mathcal{L}:= {\mathcal L}_{P}^{\epsilon'}$ is finite and $\overline{\mathcal L} \subset V$.
Moreover, we can choose $\epsilon'$ such that $\mathcal{L}$ is transverse 
to 
$\partial {\operatorname{B}}_{\epsilon'}$ by Sard's theorem.
Therefore, we get $\overline{\mathcal L} ={\mathfrak L}_{P}^{\epsilon'}$.
We know that $\mathcal L$ has a tubular neighborhood, 
since it is a Stein subvariety of the (Stein) open ball, see \cite{Forster-Ramsport:analytische, Siu:Stein}. 
That is, there is an open neighborhood $U$ of ${\mathcal L}$ and a retraction-submersion
$$
r:U\rightarrow {\mathcal L}.
$$
We can assume that the fibers of $r$ are transversal to $\mathcal{F}$, by taking a smaller $U$ if necessary. 
By transversality of $\mathcal L$  and $\partial {\operatorname{B}}_{\epsilon'}$
and finiteness of the holonomy group of $\mathcal L$,
we can argue as in Reeb's stability theorem \cite{Wu-Reeb,Camacho-Lins_Neto:foliations},
applied to the compact invariant set $\overline{\mathcal L}$,
to obtain a fundamental system of $\mathfrak{F}^{\epsilon'}$-invariant neighborhoods 
of $\overline{\mathcal L}$ in $\overline{\operatorname{B}}_{\epsilon'}$.
By considering one of such neighborhoods $W$ sufficiently small, we obtain 
a  $\mathfrak{F}^{\epsilon}$-invariant neighborhood $W \cap \overline{\operatorname{B}}_{\epsilon}$
of $\mathfrak{L}$ contained in $V$.
\end{proof}
\begin{definition}
\label{def:series}
We denote by $\hat{\mathcal O}_{n}$  and ${\mathfrak m}_{n}$ the 
ring ${\mathbb C}[[x_1, \hdots, x_n]]$
of formal power series in $n$-variables
with complex coefficients and its maximal ideal respectively. 
 Denote by $\hat{K}_{n}$ the fraction field of 
$\hat{\mathcal O}_{n}$. 
We denote by ${\mathcal O}_{n}$ the subring of $\hat{\mathcal O}_{n}$ 
of convergent power series.
\end{definition}  
\begin{definition}
\label{def:formal_vf}
We denote by $\Xn{}{n}$  (resp. 
$\Xf{}{n}$) the Lie algebra of (holomorphic) germs of vector field singular
at the origin of ${\mathbb C}^{n}$ (resp. formal vector fields).
An element $X$ of $\Xn{}{n}$ (resp. $\Xf{}{n}$)  is of the form 
\[  X = f_1 (x_1, \hdots, x_n) \frac{\partial}{\partial x_1} + \ldots +  f_n (x_1, \hdots, x_n) \frac{\partial}{\partial x_n} \]
where $f_1, \ldots, f_n \in {\mathcal O}_{n} \cap {\mathfrak m}_{n}$ 
(resp. ${\mathfrak m}_{n}$). We denote by $D_0 X$ the linear part of $X$ at ${\bf 0}$.
\end{definition}
\begin{definition}
\label{def:flow}
Given $t \in {\mathbb C}$ and $X \in \Xn{}{n}$, we denote by 
$\mathrm{exp} (t X)$ the time-$t$ flow of $X$. We define the real part $\mathrm{Re} (X)$ as the real
vector field whose time-$t$ flow is equal to $\mathrm{exp} (t X)$ for any $t \in {\mathbb R}$.
Analogously, we define the imaginary part $\mathrm{Im} (X)$ as the real
vector field whose time-$t$ flow is equal to $\mathrm{exp} (i t X)$ for any $t \in {\mathbb R}$.
\end{definition}
\begin{remark}
Consider a holomorphic vector field $X=f(z) \frac{\partial}{\partial z}$.  Then 
\[ \mathrm{Re} (X) = \mathrm{Re} (f) \frac{\partial}{\partial x} + \mathrm{Im} (f) \frac{\partial}{\partial y} \ \mathrm{and} \ 
\mathrm{Im} (X) = - \mathrm{Im} (f)\frac{\partial}{\partial x} +  \mathrm{Re} (f)  \frac{\partial}{\partial y} \]
where $z = x+i y$.
\end{remark}
\begin{definition}
\label{def:formal_dif}
We denote by $\diff{}{n}$ (resp. $\diffh{}{n}$)
the group of (holomorphic) germs of diffeomorphism 
at $\cn{n}$ (resp. formal diffeomorphisms). 
An element $\phi$ of $\diff{}{n}$ (resp. $\diffh{}{n}$)  is of the form 
\[ \phi(x_1, \ldots, x_n) = (\phi_1 (x_1, \hdots, x_n), \ldots, \phi_n (x_1, \hdots, x_n) ) \]
where $\phi_1, \ldots, \phi_n \in {\mathcal O}_{n} \cap {\mathfrak m}_{n}$ 
(resp. ${\mathfrak m}_{n}$) and the linear part $D_{0} \phi$ of $\phi$ at ${\bf 0}$ is a linear isomorphism.
In both cases the operation law is the composition.
\end{definition}
\begin{remark}
Given $t \in {\mathbb C}$ and $X \in \Xn{}{n}$, the time-$t$ flow
$\mathrm{exp} (t X)$  belongs to $\diff{}{n}$.
\end{remark}
Next, we show the finiteness of holonomy groups of leaves in the completely integrable setting.
The proof relies on results in \cite{Rib:cimpa} and more precisely that the subgroup $G$ of  
$\diffh{}{q}$ of formal diffeomorphisms
preserving generically independent germs of meromorphic map  
$\overline{f}_1, \hdots, \overline{f}_q$ is Zariski-closed (a projective limit of linear algebraic groups) and
its Lie algebra 
\[ {\mathfrak g} = \{ Z \in \Xf{}{q} : Z (\overline{f}_{j})=0 \  \forall 1 \leq j \leq q \} \]
is trivial,  inducing $G$ to be finite. 
\begin{proposition}
\label{pro:finite}
Let $\mathcal{F}$ be a CI germ of foliation.
Consider $\epsilon \in (0, \epsilon_0]$ where $\epsilon_0 >0$ is given by Proposition \ref{prop:closedleaves}.
Let $\mathcal{L}_{P}^{\epsilon}$ be a leaf of $\mathcal{F}^{\epsilon}$. 
Then the image of the holonomy representation of $\mathcal{L}_{P}^{\epsilon}$ is a finite group.
\end{proposition}
\begin{proof}
Denote $\mathcal{L} =\mathcal{L}_{P}^{\epsilon}$.
Let ${\mathcal H}_{\mathcal L}\subset \diff{}{q}$ 
be the image of the holonomy representation 
$\pi_{1} (\mathcal{L},P) \to \diff{}{q}$,  
where $\cn{q}$ is identified with a germ of transverse section to $\mathcal{F}$ at $P$.
The first integrals $f_1, \ldots, f_q$
induce germs of meromorphic functions on the transversal $\cn{q}$, 
that we denote $\bar f_1,\ldots, \bar f_q$. We still have that $\bar f_1,\ldots, \bar f_q$ are generically independent, 
i.e. $d \bar f_1 \wedge \ldots \wedge d \bar f_q \not \equiv 0$,
since we are in a transverse section. Moreover they are invariant under the action of ${\mathcal H}_{\mathcal L}$. 
This implies that ${\mathcal H}_{\mathcal L}$ is contained in the group
$$
\{ F \in \diff{}{q} : \bar f_j \circ F = \bar f_j \ \mathrm{for \ any} \ 1 \leq j \leq q \} .
$$
Since the latter group is finite in view of \cite[Proposition 3.46]{Rib:cimpa}, we are done.
\end{proof}
The next result is a direct consequence of Propositions \ref{pro:stability} and \ref{pro:finite}.
\begin{corollary}
\label{cor:stability}
  Consider a germ of CI holomorphic foliation $\mathcal{F}$. Then there exists $\epsilon>0$ such that 
  any closed leaf ${\mathfrak L}^\epsilon_P$ of $\mathfrak{F}^{\epsilon}$ has
  a fundamental system of  $\mathfrak{F}^{\epsilon}$-invariant neighborhoods 
  in $\overline{\operatorname{B}}_{\epsilon}$.
  Moreover, the result holds for any $\epsilon>0$ sufficiently small.
\end{corollary}
 

\section{The singular set and the set of separatrices}
\label{sec:infinite}
We are going to prove Theorem \ref{teo:infinite} in this section.
The last subsection will be devoted to Theorem \ref{teo:dicritical} that is a 
straightforward corollary of Theorem \ref{teo:infinite}.

Consider a foliation $\mathcal{F}$ satisfying the hypotheses of Theorem \ref{teo:infinite}.
Assume, aiming at contradiction, that there are finitely many separatrices. 
We are going to show that $\dim (\operatorname{Sing}(\mathcal{F})) \geq 1$, contradicting
that $\mathcal{F}$ has an isolated singularity.

If the hypotheses of Theorem \ref{teo:infinite} are satisfied for any open ball ${\operatorname{B}}_{\epsilon_{1}}$
then they also hold for any ${\operatorname{B}}_{\epsilon'}$ with $0 < \epsilon' < \epsilon_1$.
In particular, we can consider $\epsilon_0 >0$ such that the hypotheses hold for any 
 ${\operatorname{B}}_{\epsilon'}$ with $\epsilon' \in (0, \epsilon_{0}]$ and also if $0 < \epsilon' - \epsilon_{0} <<1$.
We proceed in several steps.

\strut

 \subsection{Existence of a separatrix}
We show
\begin{proposition}
  There is at least one separatrix.
\end{proposition}
\begin{proof}
Consider $\epsilon \in (0, \epsilon_0]$.  Let us remark that any leaf of $\mathfrak{F}^{\epsilon}$ is closed in 
$\overline{\operatorname{B}}_{\epsilon} \setminus \{ \bf{0} \}$ by hypothesis and Remark 
\ref{rem:aclosed_implies_connected}.
Suppose, aiming at contradiction, that there is no separatrix. Therefore, 
any leaf $\mathfrak{L}_{P}^{\epsilon}$ of $\mathfrak{F}^{\epsilon}$ is closed
and hence ${\bf 0} \not \in \overline{\mathfrak L}_{P}^{\epsilon}$.
By hypothesis, $\mathfrak{L}_{P}^{\epsilon}$ has a finite holonomy group. 
Thus, there exists a $\mathfrak{F}^{\epsilon}$-invariant neighborhood $W_P$ of $\mathfrak{L}_{P}^{\epsilon}$
in $\overline{\operatorname{B}}_{\epsilon}$ such that ${\bf 0} \not \in \overline{W}_{P}$ 
by Proposition \ref{pro:stability}. By compactness of $\partial {\operatorname{B}}_{\epsilon}$, there exist
$P_1, \ldots, P_k \in \partial {\operatorname{B}}_{\epsilon}$ such that 
$\partial {\operatorname{B}}_{\epsilon}$ is contained in $\cup_{j=1}^{k} W_{P_{j}}$.
Now, consider $P \in  {\operatorname{B}}_{\epsilon} \setminus \{ {\bf 0} \}$ 
sufficiently close to ${\bf 0}$ that does not belong to 
$\cup_{j=1}^{k} W_{P_{j}}$.
The leaf ${\mathfrak L}_{P}^{\epsilon}$ does not intersect $\cup_{j=1}^{k} W_{P_{j}}$, and hence 
$\partial {\operatorname{B}}_{\epsilon}$, since $\cup_{j=1}^{k} W_{P_{j}}$ is $\mathfrak{F}^{\epsilon}$-invariant.
Therefore, ${\mathfrak L}_{P}^{\epsilon}$ is a compact analytic subset 
of ${\operatorname{B}}_{\epsilon}$
of dimension $\dim (\mathcal{F}) \geq 1$ that is equal to ${\mathcal L}_{P}^{\epsilon}$.
We obtain a contradiction since such a set does not exist by the maximum modulus principle.
\end{proof}
\begin{remark}
Consider a germ ${\mathcal F}$ of CI foliation in $\cn{3}$ of codimension $2$.
Since ${\mathcal F}$ is tangent to a codimension one foliation, 
for instance $d f =0$ where $f$ is a non-constant meromorphic first integral, the existence of a separatrix can also be obtained,
in such a case, 
as a consequence of a theorem of Cerveau and Lins Neto \cite[Proposition 3]{Cerveau-Lins_Neto:cod2}.
\end{remark}  
\subsection{The neighborhood of the set of separatrices}
Recall that we are assuming that we have only finitely many separatrices through the origin, which we denote as 
$$
\Gamma_1,\Gamma_2,\ldots,\Gamma_\ell, \quad \ell\geq 1.
$$
Denote
$
\Sigma=\Gamma_1\cup\Gamma_2\cup\cdots\cup\Gamma_\ell
$.
\begin{remark}
Up to consider a smaller  $\epsilon_{0}>0$ if necessary, 
the boundaries of all the balls ${\operatorname{B}}_{\epsilon}$ 
transversely intersect each $\Gamma_i$, for $i=1,2,\ldots,\ell$ and the pair 
$(\overline{\operatorname{B}}_{\epsilon},\overline{\operatorname{B}}_{\epsilon} \cap \Sigma)$ is homeomorphic to the cone over the pair 
$(\partial {\operatorname{B}}_{\epsilon},\partial {\operatorname{B}}_{\epsilon} \cap \Sigma)$ for any
$0 <\epsilon \leq \epsilon_{0}$
by the local conic structure of analytic sets 
\cite[Lemma 3.2]{Burghelea-Verona:local_analytic}
\cite[Theorem 2.10]{Milnor:hypersurfaces}.
Therefore, ${\mathcal F}$ is topologically completely integrable.
Note that the intersection $\Sigma\cap \partial {\operatorname{B}}_{\epsilon}$ 
is a finite union of $\ell$  mutually disjoint connected manifolds of real dimension $2 \dim (\mathcal{F}) -1$.
\end{remark}

In this subsection we study a fundamental system of open neighborhood of $\Sigma$ and introduce 
transverse sections that will be key to understand their structure.

A {\em system of transverse sections} ${\mathcal T}_{\epsilon}=\{T^\epsilon_1, T_2^{\epsilon},\ldots,T_\ell^{\epsilon}\}$ for
 $0<\epsilon\leq \epsilon_0$ is the data of a $q$ dimensional section $T^\epsilon_j\subset {\operatorname{B}}_{\epsilon}$
  transverse to $\Sigma$ at a point $O^{\epsilon}_j\in \Gamma_j$, for each $j=1,2,\ldots,\ell$, in such a way that the $T_j^\epsilon$ are
   pairwise disjoint, they intersect $\Sigma$ in the single point $O^\epsilon_j$ 
   and they are transverse to $\mathcal{F}$. We assume that each
    $T^\epsilon_j$ is isomorphic to an open $q$ dimensional disk centered at $O^\epsilon_j$ and moreover the closure 
    $\overline T^\epsilon_j$ is isomorphic to the closed disc with boundary contained in 
    ${\operatorname{B}}_{\epsilon} \setminus \Sigma$.
    
 Next, we show the existence of a fundamental system of neighborhoods of $\Sigma$ in ${\operatorname{B}}_{\epsilon}$ 
 and a system of transverse sections.
\begin{definition}
\label{def:h_m}
Given a metric space $(M,d)$, we define $H(M)$ as the space of non-empty bounded closed subsets of $M$.  Given 
$A, B \in H(M)$, we define 
\[ {\mathfrak d}(A, B) = \inf \{ \epsilon \in \mathbb{R}^{+} : A \subset \cup_{P \in B}  \operatorname{B}(P; \epsilon) \ \mathrm{and} \ 
B \subset \cup_{P \in A}  \operatorname{B}(P; \epsilon)  \} . \]
Then $(H(M), {\mathfrak d})$ is a metric space \cite[Theorem 4.2]{Nadler:continuum} 
and the induced topology on $H(M)$
is the so called 
{\it Hausdorff topology}.
\end{definition}   
\begin{proposition} 
\label{pro:transv_sections}
Let us fix a system of transversal sections ${\mathcal T}_{\epsilon}$ 
for any $0<\epsilon\leq \epsilon_0$. 
We can select a fundamental system 
${\mathcal U}^\epsilon=\{U^\epsilon_m\}_{m\geq k(\epsilon)}$ of invariant open and 
connected neighborhoods of 
$\Sigma$ in $\overline{\operatorname{B}}_{\epsilon}$, such that:
\begin{enumerate}
\item For any $m\geq l\geq k(\epsilon)$, we have $U^\epsilon_m\subset U^\epsilon_{l}$.
\item  The leaves in $U^\epsilon_m$ intersect transversely the boundary $\partial {\operatorname{B}}_{\epsilon}$.
 \item The boundaries of the $T^\epsilon_j$ do not intersect $U^\epsilon_m$.
 \item Each connected component $\Delta$ of $U^\epsilon_m\cap \partial {\operatorname{B}}_{\epsilon}$ contains exactly one of the sets $\Gamma_j\cap \partial {\operatorname{B}}_{\epsilon}$ for $j=1,2,\ldots,\ell$.
 \end{enumerate}
 Moreover, if $0<\epsilon<\epsilon'\leq \epsilon_0$, we have  we have that  $U^\epsilon_m\subset U^{\epsilon'}_m$ 
 for any $m\geq \max (k(\epsilon),k(\epsilon'))$.
\end{proposition}
 \begin{proof}
 Let us consider a standard euclidean distance $\lambda_{\overline{\operatorname{B}}_{\epsilon_{0}}}$ 
 in the closed ball $\overline{\operatorname{B}}_{\epsilon_{0}}$ and consider the sequence 
 $$
 C_m=\{P\in \overline{\operatorname{B}}_{\epsilon_{0}};
 \lambda_{\overline{\operatorname{B}}_{\epsilon_{0}}}(P, \Sigma)\geq 1/m
  \}
 $$
 of compact sets 
 and let $A_m$ be the $\mathfrak{F}^{\epsilon_{0}}$-saturation of $C_m$  
 in the closed ball $\overline{\operatorname{B}}_{\epsilon_{0}}$, 
 that is $A_m$ is the union of the leaves ${\mathfrak L}_{P}^{\epsilon_0}$, for $P\in C_m$. 
 Let us see that there is a positive constant $0<\mu_m\leq 1/m$ such that 
 $\lambda_{\overline{\operatorname{B}}_{\epsilon_{0}}}(P,\Sigma)\geq\mu_m$ 
 for any $P\in A_m$. Indeed, the function
 $$
 C_m\rightarrow {\mathbb R}_{>0};\quad P\mapsto \lambda_{\overline{\operatorname{B}}_{\epsilon_{0}}}({\mathfrak L}_{P}^{\epsilon_{0}},\Sigma)
 $$
 is lower-semicontinuous, in view of Proposition \ref{pro:stability} and the closedness of the leaves. Hence,  
 it has positive minimum $\mu_m$, since $C_m$ is compact. Consider
 $$
 V_m= \overline{\operatorname{B}}_{\epsilon_{0}} \setminus A_m;
 $$
 it is a  neighborhood of $\Sigma$, since it contains the points $P$ such that
  $\lambda_{\overline{\operatorname{B}}_{\epsilon_{0}}}(P,\Sigma) <\mu_m$.
 Moreover, $V_{m} \setminus \Sigma$ is open in $\overline{\operatorname{B}}_{\epsilon_{0}}$
 by Proposition  \ref{pro:stability} and hence $V_m$ is open in $\overline{\operatorname{B}}_{\epsilon_{0}}$.
 We obtain that the family $\{V_m\}_{m\geq 1}$ is a fundamental system of 
 $\mathfrak{F}^{\epsilon_{0}}$-invariant open neighborhoods 
 of $\Sigma \cap  \overline{\operatorname{B}}_{\epsilon_{0}}$, since   
 $$
 V_m\subset \{P \in \overline{\operatorname{B}}_{\epsilon_{0}}
 ;\lambda_{\overline{\operatorname{B}}_{\epsilon_{0}}}(P,\Sigma)<1/m\}.
 $$
    Now, 
 we define $U^{\epsilon_0}_m$ as the connected component of $V_m$ containing $\Sigma$.
 It is an open subset of  $\overline{\operatorname{B}}_{\epsilon_{0}}$.
Let us consider $0<\epsilon\leq \epsilon_0$. We define $U_m^{\epsilon}$ to be the connected component of $\Sigma$ in 
 $U_m^{\epsilon_0}\cap \overline{\operatorname{B}}_{\epsilon}$. Taking $k(\epsilon) \in \mathbb{R}^{+}$ 
 big enough, we obtain the desired properties except (2) and (4).

Let us fix $0<\epsilon\leq \epsilon_0$. 
 By compactness, there is a sequence $\alpha_m\rightarrow 0$, with $\alpha_m>0$ such that
 all the points in $U_{m}^{\epsilon} \cap \partial {\operatorname{B}}_{\epsilon}$ 
 are at a distance smaller than $\alpha_m$ from $\Sigma\cap\partial {\operatorname{B}}_{\epsilon}$.
 Then, for $m>>0$  all the leaves in $U_{m}^{\epsilon}$ intersect transversely 
 the boundary $\partial {\operatorname{B}}_{\epsilon}$ providing (2). 
  In view of the transversality property, the map
 $$
 U_{m}^{\epsilon} \setminus \Sigma \rightarrow H(\overline{\operatorname{B}}_{\epsilon});\quad P\mapsto 
 {\mathfrak L}_{P}^{\epsilon}
 $$
  is continuous
  by Proposition \ref{pro:finite}.

Let us show (4). Given $m>>0$, the sets   
 \[ S_{m,j} := \{ P \in  U_{m}^{\epsilon} \cap \partial {\operatorname{B}}_{\epsilon} : 
 \lambda_{\overline{\operatorname{B}}_{\epsilon_{0}}} (P, \Gamma_{j} 
 \cap \partial {\operatorname{B}}_{\epsilon}) < \alpha_{m} \} \]
 are pairwise disjoint for $j=1, \hdots, \ell$. Fix $1 \leq j \leq \ell$. We deduce that the map
 $$
 \Lambda_{j} : U_{m}^{\epsilon} \setminus \Sigma \rightarrow H(\overline{\operatorname{B}}_{\epsilon});
 \quad P\mapsto  {\mathfrak L}_{P}^{\epsilon} \cap S_{m,j}
 $$
 is continuous.
 Moreover, $\Lambda_j (P)$ is a finite union of manifolds of real dimension $2 \dim ({\mathcal F}) - 1$
 for any $P \in U_{m}^{\epsilon} \setminus \Sigma$. Let $\mathcal{C}_j$ be the 
 connected component of $U^{\epsilon}_m \cap S_{m,j}$ containing 
 $\Gamma_j \cap \partial {\operatorname{B}}_{\epsilon}$; it 
 is a neighborhood of $\Gamma_j \cap \partial {\operatorname{B}}_{\epsilon}$
 in $\partial {\operatorname{B}}_{\epsilon}$. Note that there exists 
 $Q \in U^{\epsilon}_m \setminus \Sigma$ such that $\Lambda_j (Q) \subset \mathcal{C}_j$, for instance 
 any $Q \in U_{l}^{\epsilon} \setminus \Sigma$ with $l>>m$.
  The continuity of $\Lambda_j$ and 
   $\Lambda_{j} (Q) \subset \mathcal{C}_{j}$  imply   
 $\Lambda_{j} (U^{\epsilon}_m \setminus \Sigma) \subset  \mathcal{C}_{j}$. Since
 \[ \mathcal{C}_{j} \subset U^{\epsilon}_m \cap S_{m,j} 
 = \Lambda_{j} (U^{\epsilon}_m \setminus \Sigma) \cup 
 (\Gamma_j \cap \partial {\operatorname{B}}_{\epsilon}) \subset \mathcal{C}_j, \]
 the set $U^{\epsilon}_m \cap S_{m,j}$ is connected. This completes the proof of (4).
 \end{proof}
 Let us see that the saturation of a transverse section, together with $\Sigma$, provides a neighborhood of $\Sigma$.  
 \begin{proposition}
 \label{pro:saturated}
Consider $0 < \epsilon \leq \epsilon_0$ and $m \geq k(\epsilon)$.
Then the $\mathfrak{F}^{\epsilon}$-saturated $S_{s}^{\epsilon}$ of $T_{s}^{\epsilon}$
contains the set 
$(U_{m}^{\epsilon} \setminus \Sigma)\cup ((\Gamma_s \cap \overline{\operatorname{B}}_{\epsilon} )\setminus \{\mathbf{0}\})$
for any $s=1,2,\ldots,\ell$.
\end{proposition}
\begin{proof}
 Consider the set
$$
A=\{P\in U_{m}^{\epsilon} \setminus \Sigma;\; {\mathfrak L}_{P}^{\epsilon} \cap T_{s}^{\epsilon} \ne \emptyset\}.
$$
Note that $A=\{P\in U_{m}^{\epsilon} \setminus \Sigma;\; {\mathfrak L}_{P}^{\epsilon} \cap \overline{T}_{s}^{\epsilon} \ne \emptyset\}$
since $\partial T_{s}^{\epsilon} \cap U_{m}^{\epsilon} = \emptyset$ (Proposition \ref{pro:transv_sections}).
Since the map
 $$
 \Lambda_{j} : U_{m}^{\epsilon} \setminus \Sigma \rightarrow H( \overline{T}_{s}^{\epsilon} );
 \quad P\mapsto  {\mathfrak L}_{P}^{\epsilon} \cap T_{s}^{\epsilon}
 $$
 is continuous by Proposition \ref{pro:stability}, it follows that 
$A$ is open and closed in the connected set $U_{m}^{\epsilon} \setminus \Sigma$ and
hence $A=U_{m}^{\epsilon} \setminus \Sigma$.  
We end the proof by noting that 
$( \Gamma_s \cap \overline{\operatorname{B}}_{\epsilon}  )\setminus\{\mathbf{0}\}$ 
is the $\mathfrak{F}^{\epsilon}$-saturated of $\{O_{s}^{\epsilon} \}$.
\end{proof}
  
\subsection{Construction of holomorphic first integrals of ${\mathcal F}$} 
  In this section we are going to build a first integral in some $ {\operatorname{B}}_{\epsilon}$ whose
level sets, besides $\Sigma$, have leaves of the form ${\mathcal L}_{P}^{\epsilon}$ as connected components. 
 \begin{proposition}
 \label{pro:submersiveintegrals}
 Let $1 \leq s \leq \ell$. Then there exist $\epsilon \in (0, \epsilon_0)$  and
 a holomorphic map
 $
 H_{s}:  {\operatorname{B}}_{\epsilon}  \rightarrow T_{s}^{\epsilon_{0}}
 $
 such that 
 \begin{itemize}
 \item $H_{s}(P)\in{\mathfrak L}_{P}^{\epsilon_{0}}\cap T_{s}^{\epsilon_{0}}$ 
 for any $P\in  {\operatorname{B}}_{\epsilon}  \setminus \Sigma$;
  \item $H_{s} ({\mathcal L}_{P}^{\epsilon}) = \{ H_{s}(P) \}$ for any $P\in  {\operatorname{B}}_{\epsilon}  \setminus \Sigma$;
 \item $H_{s}(\Sigma)=\{O_{s}^{\epsilon_{0}}\}$;
 \item $H_{s}$ is a submersion on ${\operatorname{B}}_{\epsilon} \setminus \Sigma$.
 \end{itemize}
 \end{proposition}
 \begin{proof} 
  Consider $s=1$ without lack of generality. 
Denote  $T=T_{1}^{\epsilon_{0}}$ 
and consider $U := U_{m}^{\epsilon_{0}}$ for $m >>0$.
 We start by building a multivalued holomorphic map 
 $H\vert_{U \setminus \Sigma}:U \setminus \Sigma  \rightarrow T$. 
 Given $P\in U \setminus \Sigma$, consider a path 
 $\gamma : [0,1] \to {\mathfrak L}_{P}^{\epsilon_{0}}$ joining $P$ with $T$ 
 inside the leaf ${\mathfrak L}_{P}^{\epsilon_{0}}$, i.e. $\gamma(0) =P$ and $\gamma (1) \in T$. 
 Denote the corresponding holonomy map as
 $$
 \theta_{\gamma, P}: W_{\gamma,P}\rightarrow T,
 $$
where $W_{\gamma,P}$ is a suitable open neighborhood of $P$.
By construction it is a submersion such that $\theta_{\gamma, P} (Q) \in {\mathfrak L}_{Q}^{\epsilon_{0}} \cap T$
for any $Q \in W_{\gamma,P}$.
In view of the finiteness of the holonomy of ${\mathfrak L}_{P}^{\epsilon_{0}}$, 
there is an open neighborhood $W$ of $P$ and a finite set of submersive maps
$$
\theta_1,\theta_2,\ldots,\theta_k:W\rightarrow T
$$
such that for any path  $\gamma : [0,1] \to {\mathfrak L}_{P}^{\epsilon_{0}}$, 
with $\gamma(0) =P$ and $\gamma (1) \in T$, 
we have that $W\subset W_{\gamma,P}$ and there is an index $l \in \{1,2,\ldots,k\}$ with
$$
\theta_{\gamma, P}\vert_W= \theta_{l}.
$$
Moreover, the maps $\theta_j$, $j=1,2,\ldots,k$, are valid in the above sense for any point $Q\in W$ 
(use the existence of a tubular neighborhood of $\mathfrak{L}_{P}^{\epsilon_{0}}$). 
The set of holonomy maps ${\theta_1,\theta_2,\ldots,\theta_k}$ being finite, 
we can make analytic continuation of each $\theta_i$, $i=1,2,\ldots,k$ along paths in 
$U \setminus \Sigma$. 
Note that $k$ is just the cardinal of 
${\mathfrak L}_{Q}^{\epsilon_{0}} \cap T$ for a generic $Q \in U$.
   
There exists a ball ${\operatorname{B}}_{\epsilon}$ contained in $U$. 
Noting that ${\operatorname{B}}_{\epsilon} \setminus \Sigma$ is simply connected, 
since $\mathrm{cod} (\Sigma) \geq 2$, we obtain holomorphic submersions
$$
\theta_i: {\operatorname{B}}_{\epsilon} \setminus \Sigma\rightarrow T \setminus \{ O_{1}^{\epsilon_{0}} \}
$$
such that $\theta_{i}(P)\in{\mathfrak L}_{P}^{\epsilon_{0}}\cap T$
for all $P \in  {\operatorname{B}}_{\epsilon} \setminus \Sigma$ and $i=1,2,\ldots,k$.

Fix $1 \leq i \leq k$. 
By defining $\theta_{i} (P) = O_{1}^{\epsilon_{0}}$ for $P \in {\operatorname{B}}_{\epsilon} \cap \Sigma$, 
we obtain a continuous extension of 
$\theta_i$ to ${\operatorname{B}}_{\epsilon}$ by Proposition \ref{pro:transv_sections}.
We deduce that $\theta_{i}: {\operatorname{B}}_{\epsilon} \to T$ is holomorphic 
by Riemann extension theorem (cf. \cite[Theorem D.2]{Gunning1}).
Given $P\in  {\operatorname{B}}_{\epsilon} \setminus\Sigma$ and by  
considering analytic continuation along paths in 
$\mathcal{L}_{P}^{\epsilon}$, we obtain that $\theta_{i} (\mathcal{L}_{P}^{\epsilon}) = \{ \theta_{i} (P) \}$.
Indeed, if $\theta_i = \theta_{\gamma, P}$ and $\beta:[0,1] \to \mathcal{L}_{P}^{\epsilon}$ is a path 
with $\beta (1) = P$, then 
the analytic continuation of $\theta_{\gamma, P}$ along the inverse path $\overline{\beta}$ is equal to 
$\theta_{\beta \centerdot \gamma, \beta(0)}$ and in particular 
\[ \theta_{i} (\beta (0)) = \theta_{\beta \centerdot \gamma, \beta(0)} (\beta(0)) = \gamma (1) = 
 \theta_{\gamma, P} (P)  = \theta_{i} (P). \]
Now all the desired properties are satisfied for $\theta_{1}, \ldots, \theta_{k}$ and we can define $H_{1} = \theta_{1}$.
\end{proof}

\subsection{End of the proof of Theorem \ref{teo:infinite}} \label{subsection:end}
Consider the first integral $H_{1}:  {\operatorname{B}}_{\epsilon}  \rightarrow T_{1}^{\epsilon_{0}}$ given in Proposition 
\ref{pro:submersiveintegrals}. Denote $H=H_1$  and $T=T_{1}^{\epsilon_{0}}$.
 By considering coordinates $(\mathrm{x}_{1}, \ldots,  \mathrm{x}_{q})$ in $T$, 
 we obtain that $H$ is of the form $(f_{1},\ldots, f_{q})$ where $f_{1}, \ldots, f_{q}$ are holomorphic functions such that 
 the $q$-form $\omega := df_{1} \wedge \ldots \wedge d f_{q}$ satisfies 
 $\omega (P) \neq 0$
 for any $P \in {\operatorname{B}}_{\epsilon} \setminus \Sigma$ 
 since $H$ is a submersion on ${\operatorname{B}}_{\epsilon} \setminus \Sigma$ by construction 
 (Proposition \ref{pro:submersiveintegrals}).
 
 The form $\omega$ is of the form $g \omega'$ where $g \in {\mathcal O}_{n}$ 
 and the coefficients of the germ of $1$-form $\omega'$ have no common factors. 
 As a consequence of $\mathrm{Sing} ({\mathcal F}) = \{ {\bf 0} \}$, we deduce that either
 ${\bf 0}$ is an isolated singularity of $\omega$ or the germ of $\mathrm{Sing} (\omega)$ at the origin
 has dimension $n-1$.
 Since $\mathrm{Sing} (\omega) \subset \Sigma$, it follows that $\mathrm{Sing} (\omega) = \{ {\bf 0} \}$. 
 It is known that if $\omega (\bf{0})$ vanishes then 
 \[ \{ P \in   {\operatorname{B}}(\mathbf{0}; \epsilon); \; (df_{1} \wedge \ldots \wedge d f_{q})(P)=0 \}   \] 
 does not contain isolated points
 since it has dimension greater than $q-2$ and thus positive 
 \cite[Lemma 3.1.2]{Medeiros:singular}  \cite{Mal:Frob2}, contradicting $\mathrm{Sing} (\omega) = \{ {\bf 0} \}$.
  
 \subsection{The completely integrable case} \label{subsection:infinite_ci}
 First let us show Theorem \ref{teo:dicritical}. and then discuss a corollary. 
 \begin{proof}[Proof of Theorem \ref{teo:dicritical}]
 This is immediate since given $\mathcal{F}$ satisfying the
hypotheses of Theorem \ref{teo:dicritical}, it also satisfies the hypotheses of Theorem \ref{teo:infinite}
by Propositions \ref{prop:closedleaves} and \ref{pro:finite}.
 \end{proof}
 \begin{corollary}
 Let ${\mathcal F}$ be a germ of $\mathrm{(CI)}_{\mathcal{O}}$ foliation of codimension $q=2$ in 
 $\cn{n}$ with isolated singularity at ${\bf 0}$. Then there exists a dicritical hypersurface $S$. 
 More precisely, the restriction ${\mathcal F}_{|S}$ has a purely meromorphic first integral.   
 \end{corollary}
 \begin{proof}
 Consider a complete first integral $(f_1, f_2) \in {({\mathcal O}_{n} \cap {\mathfrak m}_{n})}^{2}$. 
 Note that all separatrices of ${\mathcal F}$ are contained in $(f_1, f_2)^{-1} ({\bf 0})$.
 As a consequence of the hypothesis, there exists an irreducible component $S$ of   $(f_1, f_2)^{-1} ({\bf 0})$
 of codimension $1$ containing infinitely many separatrices.
 Moreover, there exists a meromorphic first integral $h$ of ${\mathcal F}$ that is non-constant on $S$ by 
 Proposition \ref{pro_basica}. Finally $h_{|S}$ is purely meromorphic since $S$ contains infinitely
 many separatrices.
 \end{proof}
 \begin{remark}
 ${\mathcal F}_{|S}$ is a holomorphic foliation of codimension $1$ with a purely meromorphic first integral. 
 The level sets of $h_{|S}$ consists of finitely many leaves and their closures. 
 Hence, the set of separatrices of ${\mathcal F}_{|S}$ has the same cardinal as ${\mathbb C}$.
 \end{remark}

 \section{Existence of a dicritical hypersurface}
 \label{sec:cod1}
 In this section we are going to show Theorem \ref{teo:exist_cod1}.
 The proof follows many of the same steps of the proof of Theorem \ref{teo:infinite}. 
 We are going to adapt the ideas to the current setting.

 We denote 
 $S = \{ P \in  \overline{\operatorname{B}}_{\epsilon_{0}} \setminus \{ {\bf 0} \} : {\bf 0} \in \overline{\mathfrak L}_{P}^{\epsilon_{0}} \}$
 and $\Sigma = \{ {\bf 0} \} \cup S$.
  Thus $\Sigma$ is the set of separatrices. It is compact by Proposition \ref{pro:stability}. 
  Since transversality is an open property,
  the hypothesis of the theorem holds for $0< \epsilon' - \epsilon_{0} <<1$.
  Let $\Gamma$ be the isolated separatrix of ${\mathcal F}$. 
  
\subsection{The neighborhood of the set of separatrices}
Instead of a system of transverse sections, we are going to consider a transverse section 
$T:= T^{\epsilon_{0}}$ transverse to 
$\Sigma$ at a point $O^{\epsilon_{0}} \in \Gamma \setminus \{ {\bf 0} \}$.
We assume that $\overline{T} \setminus \{ {\bf 0} \} \subset {\operatorname{B}}_{\epsilon_{0}} \setminus \Sigma$, 
$T \cap \Sigma = \{ O^{\epsilon_{0}} \}$, 
that $T$ is transverse to  ${\mathcal F}$ and that $\overline{T}$ is homeomorphic to a closed ball whereas
$T$ is homeomorphic to an open ball.

We can prove the existence of a fundamental system of invariant open and connected neighborhoods of $\Sigma$ in 
$\overline{\operatorname{B}}_{\epsilon_{0}}$ with analogous properties as in Proposition \ref{pro:transv_sections}. 
We just remove property (4) in Proposition \ref{pro:transv_sections} and in property (3) we just consider $T$.
The proof is analogous to the already explained. Consider $m>>1$; the set  
\[ A=\{P\in U_{m}^{\epsilon_{0}} \setminus \Sigma;\; {\mathfrak L}_{P}^{\epsilon_{0}} \cap T \ne \emptyset\} \]
is open in $\overline{\operatorname{B}}_{\epsilon_{0}}$, 
its boundary in $\overline{\operatorname{B}}_{\epsilon_{0}}$ is contained in $\Sigma$ and 
$A$ is a union of connected components of $U_{m}^{\epsilon_{0}} \setminus \Sigma$  
by the proof of Proposition  \ref{pro:saturated}. 
Since one of the connected components of $A$ 
contains a pointed neighborhood of $O$ in $T$ and all of them contains points arbitrarily near $O$, 
it follows that $A$ is connected. 
 Moreover, $A \cap T$ is connected by arguments analogous to the ones in the proof of property (4)
 of Proposition \ref{pro:transv_sections}.

\subsection{Construction of holomorphic first integrals} 
\label{subsec:cons}
Our goal is finding a holomorphic complete first integral 
$f: U_{m}^{\epsilon_{0}} \cap {\operatorname{B}}_{\epsilon_{0}} \to T$.

Let $U$ be a open ${\mathfrak F}^{\epsilon_0}$-invariant neighborhood of ${\bf 0}$ in 
$\overline{\operatorname{B}}_{\epsilon_{0}}$ such that  
${\mathcal F}$ is transverse to $\partial  {\operatorname{B}}_{\epsilon_{0}}$
at $U \cap \partial  {\operatorname{B}}_{\epsilon_{0}}$.
For instance, we can consider $U = U_{m}^{\epsilon_{0}}$.
Denote   $O = O^{\epsilon_{0}}$.
 Analogously as in the proof of Proposition \ref{pro:submersiveintegrals},
 we obtain multivalued maps $\theta_1,\theta_2,\ldots,\theta_k: A \rightarrow T$
 such that $\theta_{j} (P)$ tends uniformly to $O$ for $1 \leq j \leq k$ and $d(P, \Sigma) \to 0$.
 We consider their restrictions $\Theta_1, \ldots, \Theta_k$ to $U \cap T$ and 
 \[ {\mathcal H}_{{\mathcal F}, \Gamma, U, T} :=\{ \Theta_1, \ldots, \Theta_k \}. \]
 Since $U \cap T$ is a pointed neighborhood of $O$ in $T$
 and $\dim (T) = q \geq 2$, it follows that, up to consider a smaller $U$ if necessary, for instance considering a bigger $m$, 
 $\Theta_j$ is univalued in $U \cap T$ and extends continuously to $O$
 by defining $\Theta_{j} (O) = O$ for any $1 \leq j \leq k$.
 Therefore $\Theta_j$ is holomorphic in $U \cap T$ for any $1 \leq j \leq k$.
 It satisfies  $\Theta_j (U \cap T) = U \cap T$ by the ${\mathfrak F}^{\epsilon_0}$ invariance of $U$ for $1 \leq j \leq k$.
 Moreover, by the construction of $\theta_1, \ldots, \theta_k$, we obtain that 
 $\Theta_{j}^{-1}$ and $\Theta_{j} \circ \Theta_{l}$ belong to ${\mathcal H} := {\mathcal H}_{{\mathcal F}, \Gamma, U, T} $
 for any choice of $1 \leq j, l \leq k$. Thus ${\mathcal H}$
 is a finite group of biholomorphisms of 
 $U \cap T$. It induces a finite subgroup of $\mathrm{Diff} (T, O)$.
 \begin{definition}
 \label{def:total} 
 We say that ${\mathcal H}_{{\mathcal F}, \Gamma, U, T} $ is the  total holonomy group   
 associated to ${\mathcal F}$, $U$, $\Gamma$ and $T$. 
 \end{definition}
 By construction of $\theta_1, \ldots, \theta_k$, we have that 
 \[ {\mathfrak L}_{P}^{\epsilon_{0}} \cap T = {\mathcal O}_{\mathcal H}(P) \]
 for any $P \in U  \cap T$, where ${\mathcal O}_{\mathcal H}(P)$ is the ${\mathcal H}$-orbit of $P$.

 Since ${\mathcal H}$ is a finite group, it is linearizable as a subgroup of  $\mathrm{Diff} (T, O)$, i.e. it is a subgroup of $\mathrm{GL} (q, {\mathbb C})$
 in some coordinates $({\rm x}_1, \ldots, {\rm x}_q)$ in $T$ centered at $O$.
 It is easy to construct linear maps $L_1, \ldots, L_q: {\mathbb C}^{q} \to {\mathbb C}$ such that for any 
 choice of  $g_1, \ldots, g_q \in {\mathcal H}$, we have
 \[ \{ L_{1} \circ g_{1} = 0\} \cap \{ L_{2} \circ g_{2} = 0\} \cap
 \ldots \cap  \{ L_{q} \circ g_{q} = 0\} = \{ {\bf 0} \} . \] 
 Then, $F_{j}:= \prod_{\Theta \in {\mathcal H}} (L_{j} \circ \Theta)$ is a ${\mathcal H}$-invariant homogeneous polynomial 
 of degree $|{\mathcal H}|$ for any $1 \leq j \leq q$ and 
 $(F_1, \hdots, F_q)^{-1} ({\bf 0}) = \{ {\bf 0} \}$.
 Moreover, $d F_{1} \wedge \ldots \wedge d F_q \not \equiv 0$ holds 
 because otherwise the generic level set of $(F_1, \ldots, F_q)$ would have positive dimension and as consequence
 every level set would have positive dimension, contradicting that the ${\bf 0}$-level is a point.

 Given $1 \leq j \leq q$, we can obtain an extension of $F_{j}$ to $A$, that we denote by $f_{j}$, 
 in such a way that is constant on leaves 
 ${\mathfrak L}_{P}^{\epsilon_{0}}$ since $F_j$ is ${\mathcal H}$-invariant. Moreover, we can extend $f_j$ to $U$
 by defining $f_{j}(P) =0$ for any $P \in U \setminus A$.
 Since the boundary of $A$ in $U$ is contained in $\Sigma$, it follows that 
 $f_j$ is continuous in $U$. Moreover, it is holomorphic
 in $f_{j}^{-1} ({\mathbb C}^{*}) \cap {\operatorname{B}}_{\epsilon_{0}}$ since such a set is contained in $A$.
 Therefore $f_j$ is holomorphic in $U \cap {\operatorname{B}}_{\epsilon_{0}}$ 
 by Rado's theorem (cf. \cite[Corollary Q.6]{Gunning1}).
 By construction $f:=(f_1, \ldots, f_q)$ is a holomorphic complete first integral of ${\mathcal F}$ and thus 
 ${\mathcal F}$ is $(\mathrm{CI})_{\mathcal{O}}$.
 \begin{remark}
 \label{rem:analytic}
 $\Sigma \cap {\operatorname{B}}_{\epsilon_{0}}$ is contained in the analytic set $V  \cap {\operatorname{B}}_{\epsilon_{0}}$
 where $V := (f_1, \ldots, f_q)^{-1}({\bf 0})$ and  
 $U \setminus V = A$. As a consequence 
 $U \setminus \Sigma$ is connected, $A = U \setminus \Sigma$ and $\Sigma = V$.
 \end{remark}
 \begin{remark}
 \label{rem:tci_non_iso}
 It is possible to define total holonomy groups in more general contexts. 
 For instance, consider 
$f_{j} (x_1, x_2, x_3, x_4) = \sum_{k=1}^{4} x_{k}^{j+1}$ for $j \in \{1,2\}$. We have
\[ df_1 \wedge df_2 = 6 \sum_{j < k} x_{j} x_{k} (x_{k} - x_{j}) dx_{j} \wedge dx_{k} . \] 
Since $\mathrm{cod} (\mathrm{Sing}(df_1 \wedge df_2 )) =3$, the foliation 
${\mathcal F}$ of codimension $2$ with first integrals $f_1$ and $f_2$ satisfies 
$\mathrm{Sing} ({\mathcal F}) = \mathrm{Sing}(df_1 \wedge df_2 )$.
By definition ${\mathcal F}$ is $\mathrm{(CI)}_{\mathcal O}$.
The singular set of ${\mathcal F}$ is the union of $15$ lines through the origin.
  We have $\mathrm{Sing}({\mathcal F}) \cap (f_1, f_2)^{-1}({\bf 0})  = \{ {\bf 0} \}$ and, moreover,  
 $(f_1, f_2)^{-1}({\bf 0})$ consists of at most $6$ separatrices.
 The importance of the 
 transversality condition in the definition of TCI is that it forces leaves containing points close to ${\bf 0}$
 to intersect the boundary $\partial {\operatorname{B}}_{\epsilon}$ of the domain of definition transversally.  
  This happens for ${\mathcal F}$ since 
 such leaves are going to intersect the boundary near $(f_1, f_2)^{-1}({\bf 0})$ that has a 
 local conic structure. 
 We can show analogues in this setting of 
 Propositions \ref{pro:transv_sections}, \ref{pro:saturated} and \ref{pro:submersiveintegrals}
 to define the total holonomy group of the foliation ${\mathcal F}$ in a transverse section to a separatrix.
 This expansion of the idea of TCI foliations to the case of non-isolated singularities will be studied in future work.
 Let us advance that there are examples of non-cyclic total holonomy groups. 
 \end{remark}
\subsection{End of the proof of Theorem \ref{teo:exist_cod1}}
It suffices to show that 
$\Sigma$ has an irreducible component of dimension $n-1$. 
Suppose, aiming at contradiction, that $\dim (\Sigma) \leq n-2$.
There exists a ball ${\operatorname{B}}_{\epsilon}$ contained in $U$ and we obtain that 
$\theta_{i}: {\operatorname{B}}_{\epsilon} \to T$ is a (univalued) holomorphic map for any $1 \leq i \leq k$
as in the proof of Proposition \ref{pro:submersiveintegrals}.
By considering coordinates $(\mathrm{x}_{1}, \ldots,  \mathrm{x}_{q})$ in $T$, we see that  
the map $H:= \theta_1$ is of the form $(f_{1}', \ldots, f_{q}')$. Moreover, it is a submersion on 
${\operatorname{B}}_{\epsilon} \setminus \Sigma$ by construction.
Now, we proceed as in subsection \ref{subsection:end}.
The $1$-form $\omega := d f_{1}' \wedge \ldots \wedge d f_{q}'$
satisfies that either $\mathrm{Sing} (\omega) = \{ {\bf 0} \}$ or $\dim (\mathrm{Sing} (\omega)) =n-1$.
The latter option can be ruled out since $\mathrm{Sing} (\omega) \subset \Sigma$. 
The former option is also impossible since $\dim (\mathrm{Sing} (\omega)) \geq q-1 \geq 1$  \cite[Lemma 3.1.2]{Medeiros:singular}.
The proof is complete.

 \section{TCI germs of vector fields}
 \label{sec:tci}
 In this section we are going to improve Theorem \ref{teo:exist_cod1} 
 for the case of vector fields. 
 Indeed, next we show  Theorem \ref{teo:characterization} 
 that provides a characterization of 
 topologically complete integrable vector fields with an isolated separatrix. 

 \begin{definition}
 We denote by ${\mathcal F}_{X}$ the foliation induced by a germ of holomorphic vector field $X$
 such that $\mathrm{cod} (\mathrm{Sing}(X)) \geq 2$.
 \end{definition}
 
First, we show the sufficient condition in Theorem \ref{teo:characterization}. 
 \begin{lemma}
 \label{lem:suff}
 Consider a germ of holomorphic foliation ${\mathcal F}$ in $\cn{n}$ 
 that is analytically conjugated to  ${\mathcal F}_{X}$ where
 \[ X= \sum_{j=1}^{n} \lambda_{j} x_j \frac{\partial}{\partial x_j}, \ \  
\lambda_1, \ldots, \lambda_{n-1} \in {\mathbb Z}^{-}, \ \ \lambda_{n} \in {\mathbb Z}^{+}.  \]
Then  ${\mathcal F}$ is topologically completely  integrable and completely integrable and has an isolated separatrix.
 \end{lemma}
 \begin{proof}
 It suffices to show the properties for ${\mathcal F}_{X}$.
 The foliation ${\mathcal F}_{X}$ is $(\mathrm{CI})_{\mathcal{O}}$ since the map  
 \[ f:= (f_1, \ldots, f_{n-1}) := (x_{1}^{\lambda_{n}} x_{n}^{-\lambda_{1}}, \ldots, x_{n-1}^{\lambda_{n}} x_{n}^{-\lambda_{n-1}}) \]
 is a holomorphic complete first integral of  $\mathcal{F}_{X}$. The set of separatrices coincide with 
 $f^{-1}({\bf 0})$, i.e. it is the union of $x_{n}=0$ and the $x_n$-axis. Indeed, the separatrices, different than the $x_n$-axis, 
are the curves admitting parametrizations of the form 
  \[ t \mapsto (c_1 t^{-\lambda_1}, \ldots, c_{n-1} t^{-\lambda_{n-1}}, 0) \]
for some $(c_{1}, \hdots c_{n}) \in ({\mathbb C}^{n})^{*}$.  Since
 \[ {X \left( \sum_{j=1}^{n-1} |x_{j}|^{2} \right)} =  \sum_{j=1}^{n-1} \lambda_{j} |x_{j}|^{2} \  \ \mathrm{and} \ \ 
  {X \left( |x_{n}|^{2} \right)} =  \lambda_{n} |x_{n}|^{2}, \]
  it follows that $X$ is transverse to $\partial {\operatorname{B}}_{\epsilon}$ at the points in  
  $x_{n}=0$ and the $x_{n}$-axis
 for $\epsilon >0$.  
  We deduce that ${\mathcal F}_{X}$ is topologically completely integrable by 
 Propositions \ref{prop:closedleaves} and  \ref{pro:finite}. 
 \end{proof}
 
  Assume, in the remainder of the section, 
that ${\mathcal F}$ is a germ of holomorphic one-dimensional foliation in $\cn{n}$ ($n \geq 3$)
that is topologically completely integrable and has an isolated separatrix $\Gamma$ and 
fix the coordinate system in the definition of TCI (cf. Definition \ref{def:tci}).
Let $X$ be a germ of vector field with isolated singularity
defining the foliation ${\mathcal F}$. Consider the notations in section \ref{sec:cod1}.  

 \begin{proposition}
 \label{pro:non-nil}
 The linear part $D_{\bf 0} X$ of $X$ at the origin is non-nilpotent.
 \end{proposition}
 \begin{proof}
  By definition of TCI there exists a decreasing sequence ${(\epsilon_{m})}_{m \geq 1}$ in ${\mathbb R}^{+}$ such that
 ${\mathcal F}$ is transverse to $\partial {\operatorname{B}}_{\epsilon_{m}}$ at the points of 
 \[ \Sigma_{m} = \{ P \in  \overline{\operatorname{B}}_{\epsilon_{m}} \setminus \{ {\bf 0} \} : 
 {\bf 0} \in \overline{\mathfrak L}_{P}^{\epsilon_{m}} \}  \cup \{ {\bf 0} \}\]
 for any $m \in {\mathbb Z}_{\geq 1}$.
 The set $\Sigma_{m} \cap  {\operatorname{B}}_{\epsilon_{m}} $ is analytic in ${\operatorname{B}}_{\epsilon_{m}} $ 
 for $m \geq 1$  by Remark \ref{rem:analytic}.
 Since $\Sigma_{m+1} \subset \Sigma_m$ for any $m \in {\mathbb Z}_{\geq 1}$, the germ
 $(\Sigma_{m}, {\bf 0})$ is independent of $m$ for any $m  >>1$.
 Since $\Sigma_{m} \cap  {\operatorname{B}}_{\epsilon_{m}} $ is analytic in 
 $ {\operatorname{B}}_{\epsilon_{m}} $ for $m  \in {\mathbb Z}_{\geq 1}$, 
 there exists $m_{0} \in {\mathbb Z}_{\geq 1}$ such that 
 $\Sigma_m = \Sigma_{m_0} \cap {\operatorname{B}}_{\epsilon_{m}} $ for any $m  \geq m_{0}$.
 Consider $\epsilon = \epsilon_{m}$ for some $m>>1$; the pair 
 $(\overline{\operatorname{B}}_{\epsilon}, \overline{\operatorname{B}}_{\epsilon} \cap \Sigma_{m_{0}})$
 is homeomorphic to the cone over  
 $(\partial {\operatorname{B}}_{\epsilon}, \partial {\operatorname{B}}_{\epsilon} \cap \Sigma_{m_{0}})$
 by the local conic structure of analytic sets \cite[Lemma 3.2]{Burghelea-Verona:local_analytic}.
 
 By Theorem \ref{teo:exist_cod1} the analytic set  $\Sigma_{m_{0}} \cap {\operatorname{B}}_{\epsilon_{m_{0}}}$ 
 has dimension $n-1 \geq 2$.
 Choose germs ${\Gamma}_{\circ}$ of
 non-isolated separatrix of ${\mathcal F}$ and $S$ of irreducible analytic 
 set of dimension $2$ such that 
 ${\Gamma}_{\circ} \subset S \subset (\Sigma_{m_{0}}, {\bf 0})$. 
 Let $\epsilon = \epsilon_{m}$, for some $m>>1$, such that 
  $(\overline{\operatorname{B}}_{\epsilon}, \overline{\operatorname{B}}_{\epsilon} \cap S)$ and 
  $(\overline{\operatorname{B}}_{\epsilon}, \overline{\operatorname{B}}_{\epsilon} \cap {\Gamma}_{\circ})$
 are homeomorphic to the cones over 
  $(\partial {\operatorname{B}}_{\epsilon}, \partial {\operatorname{B}}_{\epsilon} \cap S)$ and
  $(\partial {\operatorname{B}}_{\epsilon}, \partial {\operatorname{B}}_{\epsilon} \cap{\Gamma}_{\circ})$ respectively.
 
 Denote $\Sigma = \Sigma_m$
 and $\partial \Sigma = \partial {\operatorname{B}}_{\epsilon} \cap \Sigma$. Consider now the function  
 $$
 \rho: {\mathbb C}^{n} \to {\mathbb R}_{\geq 0};\quad 
  (x_1, \ldots, x_n)  \mapsto  \sum_{j=1}^{n} |x_{j}|^{2}
 $$
 and the tangent vector $Y(P)$ that is the gradient at $P$ of the restriction of 
 $\rho$ to ${\mathcal L}_{P}^{\epsilon_{m_0}}$. Since $X$ is transverse to $\partial {\operatorname{B}}_{\epsilon}$
 at $\partial \Sigma$, it follows that 
 $Y(P)$ points towards the exterior of ${\operatorname{B}}_{\epsilon}$ for any 
$P \in \partial \Sigma$. In particular, we can define 
maps $(a,b): \partial \Sigma \to{\mathbb S}^{1}$ 
and $c: \partial \Sigma \to {\mathbb R}^{+}$  such that 
\[ Y(P) = c(P) (a(P) \mathrm{Re} (X) (P)  + b(P) \mathrm{Im} (X) (P)) \]
for  $P \in \partial \Sigma$ (cf. Definition \ref{def:flow}).
We can extend $(a,b)$
to $\Sigma \setminus \{ {\bf 0} \}$ continuously by the local conic structure of analytic sets. 
In particular $(a,b)$ is defined in $(S \cap {\operatorname{B}}_{\epsilon}) \setminus \{ {\bf 0} \}$.
Consider the normalization $\pi: \hat{S} \to S \cap {\operatorname{B}}_{\epsilon}$; 
since $S$ is irreducible at ${\bf 0}$, it follows that $\pi^{-1} ({\bf 0})$ is a singleton. Moreover, 
we can lift $(a,b)$ to   $\hat{S} \setminus \pi^{-1} ({\bf 0})$. Let $\hat{\Gamma}_{\circ}$ be an
 irreducible curve 
such that $\pi (\hat{\Gamma}_{\circ}) = \Gamma_{\circ} \cap {\operatorname{B}}_{\epsilon}$.
The map $\pi$ induces a homeomorphism from 
$\hat{\Gamma}_{\circ} \cap \pi^{-1}(\partial {\operatorname{B}}_{\epsilon'})$ to
${\Gamma}_{\circ} \cap \partial {\operatorname{B}}_{\epsilon'}$ for any $\epsilon' \in (0, \epsilon)$.
The description of the link of a normal surface singularity by Mumford 
implies that $\hat{\Gamma}_{\circ} \cap \pi^{-1}(\partial {\operatorname{B}}_{\epsilon'})$ is a torsion element of 
$H_{1} (\hat{S} \setminus \pi^{-1} ({\bf 0}) , \mathbb{Z})$ for any $\epsilon' \in (0, \epsilon)$ \cite[page 234]{Mumford:topology_normal}.
As a consequence $(a,b)_{|\hat{\Gamma}_{\circ} \cap \pi^{-1}(\partial {\operatorname{B}}_{\epsilon'})}$ 
is homotopic to the constant map $(1,0)$ and then 
so is $(a,b)_{|{\Gamma}_{\circ} \cap \partial {\operatorname{B}}_{\epsilon'}}$ for   $\epsilon' \in (0, \epsilon)$.
By continuity $(a,b)_{|{\Gamma}_{\circ} \cap \partial {\operatorname{B}}_{\epsilon}}$ is homotopic to $(1,0)$.
Similar arguments are explored in  \cite{Ito:Poincare-Bendixson}.

Let $\gamma_{\circ} (t)$ be 
an irreducible Puiseux parametrization  
of $\Gamma_{\circ}$. 
Assume that it is defined (and injective) in a neighborhood of $\overline{D}$ where 
$D$ is homeomorphic to an open disc in ${\mathbb C}$ such that  $\overline{D}$ is homeomorphic to a closed disc, 
$\gamma_{\circ} (D) = {\operatorname{B}}_{\epsilon} \cap \Gamma_{\circ}$ and 
$\gamma_{\circ} (\partial D) = \partial {\operatorname{B}}_{\epsilon} \cap \Gamma_{\circ}$.
Denote by $Z = g(t) \frac{\partial}{\partial t}$ the vector field defined in a neighborhood of $\overline{D}$
such that 
\[ (D_{t} \gamma_{\circ}) (Z(t)) \equiv X(\gamma_{\circ} (t)). \]
The previous paragraph implies that $\mathrm{Re} (Z)_{|\partial D}$ is homotopic to $Y_{|\partial D}$ that points towards the
outside of $D$. Therefore, we obtain 
\[ \mathrm{ind}(Z, 0) = \chi (D) = 1 \]
by applying Poincar\'{e}-Hopf theorem (cf. \cite[Chapter 6]{Milnor:topology_differentiable}) to $\mathrm{Re} (Z)$. 
It follows that the vanishing order of $g(t)$ at $0$ is equal to $1$. 
Assume, without lack of generality, that $[1:0: \ldots: 0]$ is the direction in the tangent cone of $\Gamma_{\circ}$ at ${\bf 0}$.
The Puiseux parametrization $\gamma_{\circ} := (\gamma_{1}, \ldots, \gamma_{n})$ satisfies that 
the vanishing order of $\gamma_{1}$ is equal to $k$ for some $k \in {\mathbb Z}_{\geq 1}$
and $\gamma_{j} = o(t^{k+1})$ for any $2 \leq j \leq n$. Denote 
$X = \sum_{j=1}^{n} h_{j} \frac{\partial}{\partial x_{j}}$.  Since
\[ h_{j} (\gamma_{\circ} (t)) = g(t) \frac{\partial \gamma_{j}}{\partial t} (t) \]
for any $1 \leq j \leq n$, it follows that the vanishing orders satisfy
\[ \nu_{0} (h_1 (\gamma_{\circ} (t)) ) = k <  \nu_{0} (h_j (\gamma_{\circ} (t)) ) \]
for any $2 \leq j \leq n$. This implies 
\[ \lambda:= \frac{\partial h_{1}}{\partial x_{1}} ({\bf 0}) \neq 0 = \frac{\partial h_{j}}{\partial x_{1}} ({\bf 0}) \]
for any $2 \leq j \leq n$. Therefore $\lambda$ is a non-vanishing eigenvalue of $D_{0} X$ corresponding to the 
eigenvector $(1,0, \ldots, 0)$.
 \end{proof}
\begin{proof}[proof of Theorem \ref{teo:characterization}]
It suffices to show the necessary condition by Lemma \ref{lem:suff}.
Since the foliation ${\mathcal F}$ is $(\mathrm{CI})_{\mathcal{O}}$ by Theorem \ref{teo:exist_cod1} and
$\mathrm{spec} (D_{0} X) \neq \{ 0 \}$ by Proposition \ref{pro:non-nil},  it follows that $X$ is analytically conjugated 
to a vector field of the form $(1+  g(x_1, \hdots, x_{n})) X_0$, where 
\[ X_{0} = \sum_{j=1}^{n} \lambda_{j} x_{j}   \frac{\partial}{\partial x_j}   \]
and $g$ belongs to the maximal ideal of ${\mathcal O}_{n}$ by a theorem of Zhang 
\cite[Theorem 1.3]{Zhang:ana_norm2}
(cf. also \cite{Llibre-Pantazi-Walcher:first_integrals, Bruno:analytical_form} for related results), so we can assume $X=X_{0}$.
Moreover, the result of Zhang also implies that the set
\[ M_{\lambda}:= \left\{ (k_1, \ldots, k_n) \in {\mathbb Z}_{\geq 0}^{n} : \sum_{j=1}^{n} k_j \lambda_j = 0 \right\}  \]
has rank $n-1$ (cf. also the proof of Proposition \ref{pro:form_conj}). Therefore, we obtain
$(\lambda_1, \ldots, \lambda_{n}) \in {\mathbb Z}^{*}$ up to multiplication by a non-vanishing complex number.
Note that $\lambda_{j} \neq 0$ for any $1 \leq j \leq n$ since ${\mathcal F}$ has an isolated singularity at the origin.

Suppose $\lambda_1, \ldots, \lambda_{r} \in {\mathbb Z}^{-}$ and 
$\lambda_{r+1}, \ldots, \lambda_{n} \in {\mathbb Z}^{+}$ for some $0 \leq r \leq n$.
Thus
\[ \{ x_{r+1} = \ldots = x_{n} =0\} \cup  \{ x_1 = \ldots = x_{r} = 0 \} \]   
 is the union of all separatrices. Since
there exists an isolated separatrix by hypothesis, it follows that 
$r=1$ or $r=n-1$. Up to replacing $X$ with $-X$ we can suppose $r=n-1$.
\end{proof}

\appendix \section{Linearization via total holonomy group}

The goal of this appendix is providing an alternative proof of Theorem  \ref{teo:characterization}, 
and more precisely of the necessary condition, 
in which every holomorphic first integral is obtained via invariant functions of the total holonomy group 
 of a TCI vector field.

First, we will see that Proposition \ref{pro:non-nil} and techniques of jordanization of vector fields show that the foliation is formally conjugated to a 
completely integrable linear foliation (Proposition \ref{pro:form_conj}). This reduces the problem to show the next result:

\begin{theorem}[{\cite[Llibre-Pantazi-Walcher]{Llibre-Pantazi-Walcher:first_integrals}}]
\label{teo:formal_to_analytic}
Let $\mathcal{F}$ be a germ of one-dimensional  foliation in $\cn{n}$ ($n \geq 2$).
Then $\mathcal{F}$
is analytically conjugated to the foliation $\mathcal{F}_{X}$ given by a germ of vector field of the form 
\[ X= \sum_{j=1}^{n} \lambda_{j} x_j \frac{\partial}{\partial x_j}, \ \  
\lambda_1, \ldots, \lambda_{n-1} \in {\mathbb Z}^{-}, \ \ \lambda_{n} \in {\mathbb Z}^{+},   \] 
if and only ${\mathcal F}$ is formally conjugated to ${\mathcal F}_{X}$.
\end{theorem}

The  proof of Theorem \ref{teo:formal_to_analytic} in \cite{Llibre-Pantazi-Walcher:first_integrals} is based on small divisor techniques in 
\cite{Bruno:analytical_form}. Here, we introduce an interplay of formal and topological techniques to generate all 
holomorphic first integrals of the foliation, leading to an alternative proof of Theorem \ref{teo:formal_to_analytic}.

\subsection{Jordanization of formal vector fields}
 First, let us introduce the Jordan decomposition of formal vector fields, 
 a useful tool that we will use to understand the properties 
 of non-degenerated singularities of $1$-dimensional  
 $(\mathrm{CI})_{\mathcal{O}}$ foliations. Given
 $X \in \Xf{}{n}$, it can be written in a unique
 way in the form 
\[ X = X_{S} + X_{N},  \] 
where $X_{S} \in \Xf{}{n}$ is semisimple, 
i.e. of the form 
\[ X_{S} = \sum_{j=1}^{n} \lambda_{j} x_{j} \frac{\partial}{\partial x_j} \]
in some formal coordinates, $X_{N} \in \Xf{}{n}$ 
is nilpotent (i.e. $D_{0} X$ is nilpotent) and $[X_{S}, X_{N}]=0$ (cf. \cite{MarJ}).  
The decomposition is particularly simple in the ``formal" CI case.

\begin{lemma}
\label{lem:jordan}
Let $X \in \Xf{}{n}$. 
Suppose there exist $f_1, \ldots, f_{n-1} \in \hat{K}_{n}$ 
(cf. Definition \ref{def:series})
such that $d f_{1} \wedge \ldots \wedge d f_{n-1} \not \equiv 0$ and
$X(f_{j})=0$ for any $1 \leq j <n$. 
Then $X_S$ and $X_N$ are linearly dependent 
over $\hat{K}_{n}$. 
Moreover, we have $X= X_{S}$ if there exist $f \in \hat{\mathcal O}_{n} \setminus \{ {\bf 0} \}$
and $\lambda \in {\mathbb C}^{*}$ such that $X(f) = \lambda f$.
\end{lemma}
\begin{proof}
The group 
\[ G = \{ \phi \in \diffh{}{n} : f_{j} \circ \phi = f_j \ \forall 1 \leq j < n \} \]
is pro-algebraic, i.e. it is the projective limit of a sequence of linear algebraic groups \cite[Remark 3.28]{Rib:cimpa}. 
Its Lie algebra ${\mathfrak g}$ satisfies
\begin{equation}
\label{equ:lag}
 {\mathfrak g} = \{ Z \in \Xf{}{n} :  Z(f_{j})=0 \ \forall 1 \leq j < n \} . 
\end{equation}
Chevalley's theorem  implies that the Lie algebra contains the semisimple and nilpotent parts of
their elements \cite[Proposition 2.5]{JR:finite}. Thus, we obtain 
\[ X_{S}(f_j) = X_{N} (f_j) = 0 \ \  \forall 1 \leq j < n . \] 
Since $d f_{1} \wedge \ldots \wedge d f_{n-1} \not \equiv 0$, 
given $Y \in {\mathfrak g} \setminus \{0\}$, 
any formal vector field $Z$
in ${\mathfrak g}$ is of the form $g Y$ where $g \in \hat{K}_{n}$.

By construction of the Jordan decomposition of $X$ and since $f$ is an eigenvector of $X$, 
we obtain $X_{S}(f) =\lambda f$ and $X_{N}(f) =0$. 
Note that $X_{N} = g X_{S}$, 
where $g \in \hat{K}_{n}$, by the first part of 
the proof and thus we obtain $g=0$, $X_{N}=0$ and $X=X_S$.
\end{proof}  

Let us begin the proof of the necessary condition in Theorem \ref{teo:characterization}. 
Next result is a consequence of Theorem 1.3 of \cite{Zhang:ana_norm2}. We include a simple proof for the 
sake of completeness.

\begin{proposition}
\label{pro:form_conj}
Let ${\mathcal F}$ be a one-dimensional germ of TCI foliation in $\cn{n}$ with an isolated separatrix. 
Then ${\mathcal F}$ is formally conjugated to the foliation given by a vector field of the form \eqref{equ:normal_form}.
\end{proposition}
\begin{proof}
The foliation ${\mathcal F}$ is $(\mathrm{CI})_{\mathcal{O}}$ by Theorem \ref{teo:exist_cod1}.
Let $f = (f_{1}, \ldots, f_{n-1})$ be the complete holomorphic first integral of ${\mathcal F}$ constructed in 
subsection \ref{subsec:cons}. Consider the   Jordan decomposition 
$X = X_{S} + X_{N}$ of $X$.  
Note that $X_{S} \neq 0$ by Proposition \ref{pro:non-nil} 
since the eigenvalues of $D_0 X$ and $X_S$ coincide. It is of the form 
$X_{0} := \sum_{j=1}^{n} \lambda_{j} x_{j} \frac{\partial}{\partial x_j}$
in some formal coordinates that we fix from now on.
Since $X$ is a multiple of $X_{S}$ by Lemma \ref{lem:jordan}, 
${\mathcal F}$ is formally conjugated to the foliation given by $X_0$.

If there is just one non-vanishing eigenvalue of $D_{0} X$  
then ${\mathcal F}_{X}$ is  regular, contradicting the hypothesis. 
If the number $k$ of non-vanishing eigenvalues satisfies $1 < k < n$, we have that 
the singularity of $X$ is non-isolated, again providing a contradiction. Thus, we obtain 
$\lambda_{j} \neq 0$ for any $1 \leq j \leq n$. Now, 
consider the set 
\begin{equation}
\label{equ:semigroup}
 M_{\lambda} := \left\{ (k_1, \ldots, k_n) \in {\mathbb Z}_{\geq 0}^{n} : \sum_{j=1}^{n} k_j \lambda_j = 0 \right\} . 
 \end{equation}
Any monomial $\underline{x}^{\underline{k}} =x_{1}^{k_1} \ldots x_{n}^{k_n}$
that has a non-zero coefficient in $f_j$ for some $1 \leq j \leq n-1$ satisfies 
$X_{S}(\underline{x}^{\underline{k}})=0$ and hence $\underline{k} \in M_{\lambda}$. 
We claim that the ${\mathbb Q}$-vector space generated by $M_{\lambda}$ has dimension $n-1$. 
Otherwise, there is a vector field 
\[ Y = \sum_{j=1}^{n} \mu_{j} x_{j} \frac{\partial}{\partial x_j} \] 
such that $\underline{\mu}$ is not a complex multiple of $\underline{\lambda}$ and 
$\underline{k} \underline{\mu} =0$ for any $\underline{k} \in M_{\lambda}$.
In particular, we obtain $Y \in {\mathfrak g}$ (cf. Equation \eqref{equ:lag}) but this is impossible since $Y$ is not a multiple of $X_S$.
Now, $\dim_{\mathbb Q} (M_{\lambda} \otimes_{\mathbb Z} {\mathbb Q} ) =n-1$ implies  
$(\lambda_1, \ldots, \lambda_{n}) \in {\mathbb Z}^{*}$ up to multiplication by a non-vanishing complex number.

Suppose $\lambda_1, \ldots, \lambda_{r} \in {\mathbb Z}^{-}$ and 
$\lambda_{r+1}, \ldots, \lambda_{n} \in {\mathbb Z}^{+}$ for some $0 \leq r \leq n$.
There exist stable and unstable manifolds $S^{-}$ and $S^{+}$ of ${\mathcal F}$ that we can assume 
\begin{equation}
\label{equ:ss}
 S^{-} = \{ x_{r+1} = \ldots = x_{n} =0\} \ \  \mathrm{and} \ \ S^{+} =  \{ x_1 = \ldots = x_{r} = 0 \}, 
 \end{equation}   
up to a tangent to the identity change of coordinates. Since ${\mathcal F}_{|S^{-}}$ and  ${\mathcal F}_{|S^{-}}$ are formally linearizable foliations
whose linear part at the origin belong to  the Poincar\'{e} domain, both 
${\mathcal F}_{|S^{-}}$ and  ${\mathcal F}_{|S^{-}}$ are analytically linearizable. 
Therefore, $S^{-} \cup S^{+}$ is the union of all separatrices. Since 
there exists an isolated separatrix by hypothesis, it follows that 
$r=1$ or $r=n-1$. Up to replacing $X$ with $-X$ we can suppose $r=n-1$.
\end{proof}
\begin{remark}
\label{rem:ss}
 We have
$S^{-} = \{ x_{n} = 0 \}$ and $S^{+} = \cap_{1  \leq j < n} \{ x_j = 0 \}$. Up to multiply $X$ by a function of the fom 
$1 + h$, where $h$ belongs to the maximal ideal of ${\mathcal O}_{n}$, we can assume $X(x_n) = \lambda_n x_n$. Since ${\mathcal F}_{|S^{-}}$
is analytically linearizable, we can assume 
\[ ({\mathcal F}_{X})_{|S^{-}} = {\mathcal F}_{\sum_{1 \leq j \leq n-1} \lambda_{j} x_{j} \frac{\partial}{\partial x_j}} \]
up to a change of coordinates $({\bf g} (x_1, \hdots, x_{n-1}), x_n) \in \diff{}{n}$.
\end{remark}
So, as a consequence of Proposition \ref{pro:form_conj}, and in order to prove Theorem \ref{teo:characterization}, it suffices to show 
Theorem \ref{teo:formal_to_analytic}.

 \subsection{Formal transversality}

Denote 
$X_0 = \sum_{j=1}^{n} \lambda_{j} x_{j} \frac{\partial}{\partial x_j}$.  
First let us see that ${\mathcal F}$ is formally conjugated to ${\mathcal F}_{X_{0}}$ by a formal diffeomorphism that 
is transversally formal (cf. \cite[Chapter VI]{Ban}) along $S^{-} \cup S^{+}$ (see Equation \eqref{equ:ss}).
This type of approach was used in \cite{Chaves:holonomy_equivalence} to improve the theorems
that relate analytic conjugacy of holonomy maps and analytic conjugacy of certain singularities of vector fields
 \cite{Eli-Ilya}  (cf. \cite{Reis:equivalence}).
 
In order to read this subsection it is not necessary to know what formal transversality is. Proposition \ref{pro:formal_trans}
is the unique result that will be used in the following.
\begin{proposition}
\label{pro:formal_trans}
Given $k \in {\mathbb Z}_{\geq 1}$, the foliation ${\mathcal F}$ is analytically conjugated to a foliation given by a vector field of the form 
\begin{equation}
\label{equ:inter}
X = \sum_{j=1}^{n-1} (\lambda_{j} x_{j} + f_{j}) \frac{\partial}{\partial x_j} + \lambda_{n} x_{n} \frac{\partial}{\partial x_{n}} 
\end{equation}
where $f_{j} \in (x_1, \ldots, x_{n-1})^{k} \cap (x_{n})^{k}$ for any $1 \leq j <n$.
\end{proposition}
\begin{proof}
We can assume ${\mathcal F} = {\mathcal F}_{X}$ where we are in the setting of Remark \ref{rem:ss}.
Since $X_0$ is $\mathrm{(CI)}_{\mathcal O}$ and $X(x_n) = \lambda_{n} x_n$, 
we obtain $X=X_{S}$ by Lemma \ref{lem:jordan}. Since $X$ is formally conjugated to $X_0$,  
we can suppose $j^{k'} X = X_0$ for some $k'>>1$, up
to an analytic change of coordinates that preserves the previous properties.
Denote $I_{-} = (x_n)$ and $I_{+} = (x_1, \ldots, x_{n-1})$.
We say that $X$ is of type $(r,s)$ if it is of the form 
\eqref{equ:inter}
with $f_{j} \in  I_{-}^{r} \cap I_{+}^{s} \cap  {\mathfrak m}^{k' + 1}$ 
for  any $1 \leq j < n$.
By Remark \ref{rem:ss},  $X$ is of type $(1,1)$.

Since a vector field of the form 
\[ X' =  \sum_{j=1}^{n-1} \left(\lambda_{j} x_{j} + \sum_{k=1}^{n-1} a_{jk} (x_{n}) x_{k} \right) \frac{\partial}{\partial x_j} + \lambda_{n} x_{n} \frac{\partial}{\partial x_{n}},   \]
where $a_{jk} \in (x_{n})^{k'}$ for all $1 \leq j,k < n$, has no resonances when $k'>>1$, it is analytically conjugated to $X_0$
by a change of coordinates of the form 
\[ \left( x_{1} + \sum_{k=1}^{n-1} b_{1k} (x_{n}) x_{k} , \ldots,  x_{n-1} + \sum_{k=1}^{n-1} b_{n-1 \, k} (x_{n}) x_{k} , x_{n} \right) \]
by the theory of Fuchsian singularities (cf. \cite[Theorems 16.15 and 16.16]{Ilya-Yako}). By considering $X'$ such that $X-X' \in I_{+}^{2}$, we obtain that the 
vector field $X$ is of type $(1,2)$ in these coordinates.

Now, let us see that we can make $X$ of type $(r+1, s)$ if $X$ is of type $(r, s)$, 
for $1 \leq r < k$ and $2 \leq s \leq k$, up
to an analytic change of coordinates. We have 
\[ f_j = x_{n}^{r} a_j (x_1, \ldots, x_{n-1}) + O(x_{n}^{r+1}) \]
where
\[ a_{j} = \sum_{\underline{k}} a_{j, \underline{k}} x_{1}^{k_1} \ldots x_{n-1}^{k_{n-1}} \in I_{+}^{s} \cap {\mathfrak m}^{k' + 1 -r} \]
for any $1 \leq j < n$. We define
\[ b_{j, \underline{k}}  = 
 \frac{- a_{j, \underline{k}} }{\lambda_j -\sum_{l=1}^{n-1} k_{j} \lambda_{j} - r \lambda_n}  . \]
Notice that $a_{j, \underline{k}} \neq 0$ implies 
$k_{1} + \ldots + k_{n-1} \geq k' + 1 -r$ and 
thus the denominator does not vanish if 
$k'>>1$. Moreover, since the modulus of the denominator is bounded by below by $1$, 
it follows that 
\[ Y := \sum_{j=1}^{n-1} \left( \sum_{\underline{k}} b_{j, \underline{k}} x_{1}^{k_1} \ldots x_{n-1}^{k_{n-1}}  \right) x_{n}^{r}  
\frac{\partial}{\partial x_j} \in \Xn{}{n}.   \]
By construction it satisfies
\[ [Y, X_0] =  - \sum_{j=1}^{n-1} a_{j} (x_1, \ldots, x_{n-1}) x_{n}^{r}  \frac{\partial}{\partial x_j} .  \]
Consider the diffeomorphism $\phi = \mathrm{exp} (Y)$ (cf. Definition \ref{def:flow}), it satisfies
\[ \phi^{*} X = X + [Y,X] + \frac{1}{2} [Y,[Y,X]] + \frac{1}{3!} [Y,[Y,[Y,X]]] + \ldots =\]
\[ X  + [Y,X_{0}] + [Y, X - X_{0}] + \frac{1}{2} [Y,[Y,X]] + \frac{1}{3!} [Y,[Y,[Y,X]]] + \ldots \]
and\ $\phi^{*} X$ is of type $(r+1, s)$. 
Finally, the transition from type $(r,s)$ to type $(r,s+1)$, for 
$1 \leq r \leq k$ and $1< s<k$, is proved analogously to 
the transition from type $(r,s)$ to type $(r+1,s)$.
\end{proof}
\begin{corollary}
${\mathcal F}$ is conjugated to ${\mathcal F}_{X_0}$ by a formal diffeomorphism that is
transversally formal along $S^{-} \cup S^{+}$.
\end{corollary}
\begin{proof}
Let us consider the notations and definitions in Proposition \ref{pro:formal_trans}.
We can suppose that $X$ is of type $(2,2)$. 
Consider 
$S:= (S^{+} \cup S^{-}) \cap {\operatorname{B}}_{\epsilon}$ for some $\epsilon >0$
so that $X$ is defined in a neighborhood of $S$.
It suffices to show that there exist  
$X_k \in \Xn{}{n}$ and 
$\phi_{k} \in \diff{}{n}$, for any $k \geq 2$, such that 
\begin{itemize}
\item $X_{2} = X$ and $X_{k+1}={(\phi_{k})}_{*} X_{k}$ for any $k \geq 2$;
\item $X_{k}$ is of type $(k, k)$ for any $k \geq 2$;
\item $X_{k}$ and $\phi_{k}$ are 
defined in a neighborhood of $S$ and
$S \subset \mathrm{Fix} (\phi_k)$ for any $k \geq 2$;
\item $\phi_{k}$ is of type $(k,k)$, i.e.  
\[ \phi_{k} (x_1, \ldots, x_n) = (x_1 + f_1 , 
\ldots, x_{n-1} + f_{n-1}, x_n), \]
where $f_1, \ldots, f_{n-1} \in I_{+}^{k} \cap I_{-}^{k}$ for any $k \geq 2$.
\end{itemize}
Then there is a unique 
$\hat{\phi} \in \diffh{}{n}$
with $j^{k-1} \hat{\phi} = j^{k-1} (\phi_{k} \circ \ldots \circ \phi_2)$
for any $k \geq 2$. Moreover, by construction $\hat{\phi}$ is transversally formal 
along $S$ and conjugates $X$ and $X_0$.

We can construct $X_{k+1}$ and $\phi_k$ by induction on $k$.
By the proof of Proposition \ref{pro:formal_trans}, 
$\phi_k$ is the composition of three local diffeomorphisms:
the first one $\psi_{1}$ gives  $j^{k'} ((\psi_{1})_{*} X_k) = X_0$
for $k'>>1$, and $\psi_2$ and $\psi_{3}$ guarantee that 
$(\psi_{2} \circ \psi_{1})_{*} X_k$ and 
$(\psi_{3} \circ \psi_{2} \circ \psi_{1})_{*} X_k$ are of types $(k+1,k)$
and $(k+1, k+1)$ respectively.
It is easy to see that 
$\psi_j$ is defined in a neighborhood of $S$,
$S \subset \mathrm{Fix} (\psi_j)$
and $\psi_{j}$ is of type $(k,k)$ for any $1 \leq j \leq 3$. 
Thus $\phi_k = \psi_{3} \circ \psi_{2} \circ \psi_{1}$ is the desired diffeomorphism.
\end{proof}

\subsection{Generating first integrals}

Finally, we get to the topological part of the proofs of Theorems  \ref{teo:characterization} and  \ref{teo:formal_to_analytic}.
We will compute the total holonomy group at a transverse section $T = \{x_{n} = c_0\}$ for some $c_0 \in {\mathbb C}^{*}$. 
By considering well chosen invariant functions by the action of the total holonomy group, we obtain a natural candidate for
the analogue for ${\mathcal F}$ of the 
the complete first integral  $(x_{1}^{\lambda_{n}} x_{n}^{-\lambda_{1}}, \ldots, x_{n-1}^{\lambda_{n}} x_{n}^{-\lambda_{n-1}})$ 
of ${\mathcal F}_{X_{0}}$. Formal transversality, and more precisely Proposition \ref{pro:formal_trans},
allows us to control such first integrals in the neighborhood of $x_{n}=0$.

Consider the holonomy map $\Theta$ associated to the isolated separatrix of ${\mathcal F}$. 
It is the holonomy map associated to 
the path $\gamma:[0,1] \to S^{+} \setminus \{ {\bf 0} \}$ defined by 
$\gamma (t) = (0, \ldots, 0, c_0 e^{2 \pi i t})$. The next result is a consequence of the formal transversality property 
provided by Proposition \ref{pro:formal_trans}.
\begin{lemma}
The holonomy map $\Theta$ associated to the isolated separatrix of ${\mathcal F}$ has finite order.
\end{lemma}
\begin{proof}
We can suppose that ${\mathcal F}= {\mathcal F}_{X}$ where $X$ is of the form \eqref{equ:inter}
where $f_{j} \in (x_1, \ldots, x_{n-1})^{k} \cap (x_{n})^{k}$ for any $1 \leq j <n$ and some $k \in {\mathbb Z}_{\geq 2}$ by 
Proposition \ref{pro:formal_trans}.
The holonomy map $\Theta$ is of the form 
\[ \Theta = (\Theta_1, \hdots, \Theta_{n-1}) (x_1, \ldots, x_{n-1}) =
{ \mathrm{exp} \left( \frac{2 \pi i}{\lambda_{n}} X \right)}_{| x_{n}=c_0}   \]
and satisfies 
$\Theta_{j} - e^{\frac{2 \pi i \lambda_j}{\lambda_{n}}} x_j \in I_{+}^{k}$ for any $1 \leq j <n$.
Note that $\Theta^{\lambda_n}$ is tangent to the identity and its order of tangency with the identity is at least $k$. 
Since $k$ is arbitrary and the order of 
tangency with the identity is an analytic invariant, we deduce that $\Theta$ has finite order. 
In particular, we obtain $\mathrm{exp}(2 \pi i X) = \mathrm{id}$.
\end{proof}
Fix  a transverse section $T = \{x_{n} = c_0 \}$, with $0 < c_0 < \epsilon <<1$,  and let $\Gamma$ be the $x_{n}$-axis.  
Next, we describe the total holonomy group (cf. Definition \ref{def:total}).
\begin{lemma}
${\mathcal F}$ is topologically completely integrable and 
its total holonomy group ${\mathcal H}_{{\mathcal F}, \Gamma, U, T}$ is equal to the finite group $\langle \Theta \rangle$.
\end{lemma}
\begin{proof}
We can suppose that $X$ is of the form \eqref{equ:inter} by Proposition \ref{pro:formal_trans} for some $k \geq 2$.
Denote $m = \min (-\lambda_1, \ldots, -\lambda_{n-1}) $,   $\underline{x}^{-} = (x_1, \ldots, x_{n-1})$
and $\underline{x} = (\underline{x}^{-} , x_{n})$.
We have
\begin{equation}
\label{equ:aux_tr}
X ( |x_{n}|^{2}  ) = \lambda_{n}  |x_{n}|^{2} \ \mathrm{and} \   
X (|\underline{x}^{-}|^{2}) = \sum_{j=1}^{n-1} \lambda_j |x_{j}|^{2} + O(|\underline{x}^{-}|^{3}) 
\end{equation}
and as a consequence we get
\begin{equation}
\label{equ:aux_dec}
\mathrm{Re} (X) \left( \ln |\underline{x}^{-}|^{2}  + \frac{m}{4 \lambda_{n}} \ln |x_{n}|^{2} \right) <0 . 
\end{equation}
Thus $|\underline{x}^{-}||x_{n}|^{\frac{4}{m \lambda_{n}}} \circ \mathrm{exp} (t X)$ 
is a decreasing function of $t \in {\mathbb R}$.

Consider the foliation ${\mathcal F}_{\partial}$ induced by ${\mathcal F}$ in $\partial {\operatorname{B}}_{\epsilon}$.
There exists an open ${\mathcal F}_{\partial}$-invariant neighborhood $W$ of 
$\Gamma \cap  \partial {\operatorname{B}}_{\epsilon}$ in $\partial {\operatorname{B}}_{\epsilon}$ 
since $\Theta$ has finite order. 
All leaves of ${\mathcal F}_{\partial}$ in $W$ are closed. 
From now on, we consider that $W$ is a sufficiently small neighborhood of 
$\Gamma \cap  \partial {\operatorname{B}}_{\epsilon}$.

Let $P \in W \setminus \Gamma$. Let $c_{P}$ be the minimum real positive number $c$ such that 
$\mathrm{exp}(- c X) (P) \not \in {\operatorname{B}}_{\epsilon}$.
Property \eqref{equ:aux_dec} implies that $c_{P}$ is well-defined and, moreover, that 
$A_{P}:=\mathrm{exp}(- c_{P} X) (P)$ is in a small neighborhood of $S^{-} \cap \partial {\operatorname{B}}_{\epsilon}$ 
The real flow of $X$ is transverse to $\partial {\operatorname{B}}_{\epsilon}$
at both $P$ and $A_{P}$ by Property \eqref{equ:aux_tr}. Thus $c_{P}$ depends continuously on $P \in W \setminus \Gamma$. 
Consider a leaf ${\mathcal L}_{\partial}$ of ${\mathcal F}_{\partial}$ through $P \in W \setminus \Gamma$.
The previous discussion implies that
\[ {\mathfrak L}_{P}^{\epsilon} = \cup_{Q \in {\mathcal L}_{\partial}} \{ \mathrm{exp}(t X) (Q) : t \in [-c_{Q},0] \}. \] 
is a closed leaf of ${\mathfrak F}^{\epsilon}$ for any $P \in W \setminus \Gamma$.
In particular, the set 
\[ U_{W} = (S^{-} \cap  \overline{\operatorname{B}}_{\epsilon}) \cup \cup_{P \in W} {\mathfrak L}_{P}^{\epsilon} \]
is an open ${\mathfrak F}^{\epsilon}$-invariant neighborhood of $S^{-} \cup S^{+}$ in 
$\overline{\operatorname{B}}_{\epsilon}$
such that $X$ is transverse to $\partial {\operatorname{B}}_{\epsilon}$ at $U_{W} \cap \partial {\operatorname{B}}_{\epsilon}$.
By considering any open ${\mathcal F}_{\partial}$-invariant neighborhood of 
$\Gamma \cap  \partial {\operatorname{B}}_{\epsilon}$ in $\partial {\operatorname{B}}_{\epsilon}$, we obtain 
a fundamental system of ${\mathcal F}^{\epsilon}$-invariant neighborhoods of $S^{-} \cup S^{+}$
in $\overline{\operatorname{B}}_{\epsilon}$. Denote  $U:=U_{W}$.


Any path $\gamma:[0,1] \to {\mathfrak L}_{P}^{\epsilon} \cap U$, with $P \not \in S^{-}$, 
can be deformed continuously in ${\mathfrak L}_{P}^{\epsilon}$
into a path $\overline{\gamma}:[0,1] \to {\mathfrak L}_{P}^{\epsilon} \cap W$.
Indeed given any $s \in [0,1]$, there exists $t_{s} \in {\mathbb R}_{\geq 0}$ such that 
$\mathrm{exp}(t X) (\gamma (s)) \in U \setminus W$ for any $0 \leq t < t_s$ and 
$\mathrm{exp}(t_{s} X) (\gamma (s)) \in W$. The transversality of $X$ with $\partial {\operatorname{B}}_{\epsilon}$ in 
$U$ implies that $s \mapsto t_s$ is continuous in $[0,1]$. The path $\overline{\gamma}$ is defined by 
$\overline{\gamma} (s) = \mathrm{exp}(t_s X) (\gamma (s))$ for $s \in [0,1]$.
Therefore,  the holonomy group of any leaf is finite since it can be considered as a subgroup of $\langle \Theta \rangle$.
Moreover, the total holonomy group ${\mathcal H}_{\mathcal F, \Gamma, U, T}$ coincides with $\langle \Theta \rangle$ 
and hence is finite.
We obtain that the foliation ${\mathcal F}$ is topologically completely integrable. 
\end{proof}
\subsubsection{End of the proof of Theorem \ref{teo:formal_to_analytic}}
Suppose $\gcd (\lambda_1, \ldots, \lambda_n) =1$ without lack of generality.
We can assume that  $X$ is of the form \eqref{equ:inter}  for some $k >>1$ fixed by Proposition \ref{pro:formal_trans}.
Given $1 \leq j < n$, we define 
\[ r_j = \lambda_n / \gcd (\lambda_j, \lambda_n) \in \mathbb{Z}^{+} \ \ \mathrm{and} \ \  
s_j = -\lambda_j / \gcd (\lambda_j, \lambda_n) \in \mathbb{Z}^{+}. \] 
The function $g_j = x_{j}^{r_j} x_{n}^{s_j}$ is a first integral of 
$X_0$ for any $1 \leq j < n$.

The holonomy map $\Theta$ is of the form $\Theta = (\Theta_1, \hdots, \Theta_{n-1})$
where $\Theta_{j} - e^{\frac{2 \pi i \lambda_j}{\lambda_{n}}} x_j \in I_{+}^{k}$ for any $1 \leq j <n$ and 
$I_{+}=(x_1, \ldots, x_{n-1})$. It has order $\lambda_{n}$. It can be linearized in $\diff{}{n-1}$
by a map of the form 
\[ G = \frac{1}{\lambda_{n}} \sum_{j=0}^{\lambda_{n}-1} (D_{0} \Theta)^{-j} \circ \Theta^{j}. \]
Indeed, we obtain $D_0 G = \mathrm{id}$, $D_{0} \Theta \circ G = G \circ \Theta$ and 
$G- \mathrm{id} \in I_{+}^{k} \times \ldots \times  I_{+}^{k}$.

Fix $1 \leq j < n$.
The equality $x_{j}^{r_j} \circ D_{0} \Theta = x_{j}^{r_j}$ implies
\[ (x_{j}^{r_j} \circ G) \circ \Theta = x_{j}^{r_j} \circ G . \] 
Note that $x_{j}^{r_j} \circ G - x_{j}^{r_j} \in I_{+}^{k}$.
Since $\langle \Theta \rangle$ is the total holonomy group of ${\mathcal F}$, we can construct a first integral
$F_j$ of $X$ in a neighborhood of ${\bf 0}$ such that $(F_{j})_{|T} \equiv  c_{0}^{s_{j}} (x_{j}^{r_j} \circ G)$. It satisfies
\[ F_{j} = x_{j}^{r_j} a_{j} (x_{n}) + h_j (x_1, \hdots, x_n) \]
where $a_{j} \not \equiv 0$ and $h_j \in I_{+}^{k}$ for any $1 \leq j < n$. Moreover, by construction 
$F_{j}$ is of the form $f_{j}^{r_j} x_{n}^{d_j}$ where $f_{j}  \in {\mathcal O}_{n}$ is irreducible and 
$d_{j} \in {\mathbb Z}_{\geq 0}$.
Since $r_{j} < k$, all non-vanishing monomials $x_{j}^{r_j} x_{n}^{d}$ of $x_{j}^{r_j} a_{j} (x_{n})$ are 
resonant, i.e. $r_j \lambda_j + d \lambda_n =0 $ and thus $d = s_j$. In particular, we obtain $a_j \equiv x_{n}^{s_j}$. 

Consider a non-vanishing monomial $x_1^{k_1} \ldots x_{n}^{k_n}$ of $h_j$. If $k_n \geq k$ then $k_n > s_j$.
Otherwise the monomial is resonant, i.e. $k_1 \lambda_1 + \ldots + k_n \lambda_n = 0$, 
since $X$ is of the form \eqref{equ:inter}. But this implies 
\[ k_n = \sum_{l=1}^{n-1} -\frac{\lambda_l}{\lambda_n} k_l  \geq (k_1 + \ldots + k_{n-1})
\min_{1 \leq l < n} \frac{-\lambda_{l}}{\lambda_{n}}  \geq k \min_{1 \leq l < n} \frac{-\lambda_{l}}{\lambda_{n}} >s_j . \]
We deduce
\[ F_{j} = f_{j}^{r_j} x_{n}^{d_j} = (x_{j}^{r_{j}} + h.o.t.) x_{n}^{s_j} \]
and hence $d_{j} = s_{j}$ and $f_{j} = x_j + h.o.t$. 

Resuming, $(y_1, \hdots, y_n) = (f_1, \ldots, f_{n-1}, x_n)$ is a system of coordinates in which 
$y_{j}^{r_j} y_{n}^{s_j}$ is a first integral of ${\mathcal F}$ for any $1 \leq j <n$. 
Therefore ${\mathcal F}$ is induced by the vector field $\sum_{j=1}^{n} \lambda_j y_j \frac{\partial}{\partial y_j}$.
\begin{remark}
The semigroup $M_{\lambda}$ (cf. Equation \eqref{equ:semigroup}) has a finite generator system $S_{\lambda}$.
Considering $X$ of the form  \eqref{equ:inter}  for some $k >>1$ fixed, we obtain the first integral  
$f_{1}^{k_{1}} \hdots f_{n-1}^{k_{n-1}} x_{n}^{k_{n}}$of ${\mathcal F}$  by applying the previous construction to the 
$\langle \Theta \rangle$-invariant function $(x_{1}^{k_1} \hdots x_{n-1}^{k_{n-1}} c_{0}^{k_{n}}) \circ G$
for any $(k_1, \hdots, k_{n}) \in S_{\lambda}$. Thus, the construction provides the ring
\[ {\mathbb C} \{ f_{1}^{k_{1}} \hdots f_{n-1}^{k_{n-1}} x_{n}^{k_{n}} : (k_1,\hdots,k_n) \in S_{\lambda} \} \]
of holomorphic first integrals of ${\mathcal F}$.
\end{remark}
\subsubsection{End of the proof of the necessary condition in Theorem  \ref{teo:characterization}}
It is an immediate consequence of Proposition \ref{pro:form_conj} and Theorem  \ref{teo:formal_to_analytic}.

\bibliography{rendu.bib}
\end{document}